\documentclass{article}

     \PassOptionsToPackage{numbers, compress}{natbib}

     \usepackage[preprint]{neurips_2024}

\usepackage[utf8]{inputenc} 
\usepackage[T1]{fontenc}    
\usepackage{hyperref}       
\usepackage{url}            
\usepackage{booktabs}       
\usepackage{amsfonts}       
\usepackage{nicefrac}       
\usepackage{microtype}      
\usepackage{xcolor}         

\usepackage{bm}               
\usepackage{bbm}

\usepackage{amsthm}
\usepackage{amscd}
\usepackage{amsmath}
\usepackage{amssymb}

\usepackage{mathtools}
\usepackage{mathrsfs}

\usepackage{tikz}
\usepackage{graphicx}
\usepackage{epstopdf}

\usepackage{caption}
\usepackage{subcaption}

\usepackage[capitalise]{cleveref}

\usepackage{times}
\usepackage{xspace}
\usepackage{enumitem}
\usepackage{tabularx}

\usepackage{algorithm}
\usepackage{algorithmic}

\usepackage{comment}
\usepackage[textsize=tiny]{todonotes}

\newcommand{\wMSE}{ \mathsf{E} }

\theoremstyle{plain}
\newtheorem{theorem}{Theorem}[section]

\newtheorem{lemma}[theorem]{Lemma}
\newtheorem{corollary}[theorem]{Corollary}
\theoremstyle{definition}
\newtheorem{definition}[theorem]{Definition}
\newtheorem{assumption}[theorem]{Assumption}
\newtheorem{problem}[theorem]{Problem}
\theoremstyle{remark}
\newtheorem{remark}[theorem]{Remark}

\newcommand{\tk}[1]{#1}

\newcommand{\Htwo}{\mathit{H_2}}
\newcommand{\Hinf}{\mathit{H_\infty}}

\newcommand{\Ss}{\mathcal{S}}

\newcommand{\w}{\mathbf{w}} 
\newcommand{\s}{\mathbf{s}} 
\newcommand{\hats}{\widehat{\s}}
\newcommand{\y}{\mathbf{y}}
\newcommand{\ee}{\mathbf{e}}
\newcommand{\vv}{\mathbf{v}}
\renewcommand{\d}{\mathsf{d}}

\newcommand{\G}{\mathcal{G}} 
\newcommand{\K}{\mathcal{K}} 
\newcommand{\T}{\mathcal{T}} 

\newcommand{\HH}{\mathcal{H}}
\newcommand{\M}{\mathcal{M}} 
\renewcommand{\L}{\mathcal{L}} 
\newcommand{\CC}{\mathcal{C}} 
\newcommand{\I}{\mathcal{I}} 
\newcommand{\U}{\mathcal{U}} 

\newcommand{\W}{\mathscr{W}} 
 
\newcommand{\causal}{\mathscr{K}}

\newcommand{\trig}{\mathscr{T}}

\newcommand{\BW}{\mathsf{BW}} 
\newcommand{\Was}{\mathsf{W_2}}
\newcommand{\HS}{2} 

\newcommand{\op}{\infty} 

\newcommand{\ejw}{\e^{{j}\omega}}

\DeclareMathOperator*{\argmax}{arg\,max}

\newcommand{\subjto}{\,\mathrm{s.t.}\,}

\newcommand{\defeq}{\coloneqq}                      

\newcommand{\inn}{\!\in\!}
\newcommand{\+}{\!+\!}
\renewcommand{\-}{\!-\!}
\renewcommand{\=}{\!=\!}

\renewcommand{\>}{\!>\!}

\newcommand{\ie}{\textit{i.e.}}

\newcommand{\eps}{\varepsilon}          

\newcommand{\N}{\mathbb{N}}     
\newcommand{\R}{\mathbb{R}}     
\newcommand{\C}{\mathbb{C}}     

\newcommand{\Sym}{\mathbb{S}}   

\newcommand{\Prob}{\mathscr{P}} 
\newcommand{\TT}{\mathbb{T}}     

\newcommand{\pr}[1]{\left({#1}\right)}          
\newcommand{\br}[1]{\left[{#1}\right]}          
\newcommand{\cl}[1]{\left\{{#1}\right\}}        

\newcommand{\abs}[1]{\vert{#1}\vert}                    
\newcommand{\Abs}[1]{\left\vert{#1}\right\vert}         

\newcommand{\norm}[2][\text{}]{\Vert{#2}\Vert_{#1}}     
\newcommand{\Norm}[2][\text{}]{\left\Vert{#2}\right\Vert_{#1}} 

\newcommand{\clf}[1]{\mathcal{#1}} 

\newcommand{\tr}{\operatorname{tr}}         
\newcommand{\Tr}{\operatorname{Tr}}         

\newcommand{\inv}{\mathrm{\-1}}             

\newcommand{\psdgeq}{\succcurlyeq} 
\newcommand{\psdg}{\succ}          

\newcommand{\E}{\operatorname{\mathbb{E}}} 
\renewcommand{\Pr}{\operatorname{\mathbb{P}}} 

\newcommand{\sampled}[1][\text{}]{\stackrel{#1}{\sim}}

\newcommand{\normal}{\operatorname{\mathcal{N}}} 

\newcommand{\beq}[1]{\begin{align*}\label{eq:#1}}
\newcommand{\eeq}{\end{align*}}

\newcommand{\suml}{\sum\nolimits}

\newcommand{\e}{\mathrm{e}}             
\newcommand{\half}{\frac{1}{2}} 

\makeatletter
\newcommand{\xMapsto}[2][]{\ext@arrow 0599{\Mapstofill@}{#1}{#2}}
\def\Mapstofill@{\arrowfill@{\Mapstochar\Relbar}\Relbar\Rightarrow}
\makeatother

\title{Distributionally Robust Kalman Filtering over Finite and Infinite-Horizon}

\author{%
  Taylan Kargin$^*$ \\
  Caltech\\
  Pasadena, CA 91125 \\
  \texttt{tkargin@caltech.edu} \\
   \And
  Joudi Hajar$^*$ \\
  Caltech\\
  Pasadena, CA 91125 \\
  \texttt{jhajar@caltech.edu} \\
   \AND
  Vikrant Malik$^*$ \\
  Caltech\\
  Pasadena, CA 91125 \\
  \texttt{vmalik@caltech.edu} \\
   \And
  Babak Hassibi \\
  Caltech\\
  Pasadena, CA 91125 \\
  \texttt{hassibi@caltech.edu} \\
}

\begin{document}

\maketitle

\begin{abstract}
This paper investigates the distributionally robust filtering of signals generated by state-space models driven by exogenous disturbances with noisy observations in finite and infinite horizon scenarios. The exact joint probability distribution of the disturbances and noise is unknown but assumed to reside within a Wasserstein-2 ambiguity ball centered around a given nominal distribution. We aim to derive a causal estimator that minimizes the worst-case mean squared estimation error among all possible distributions within this ambiguity set. We remove the iid restriction in prior works by permitting arbitrarily time-correlated disturbances and noises. In the finite horizon setting, we reduce this problem to a semi-definite program (SDP), with computational complexity scaling with the time horizon. For infinite horizon settings, we characterize the optimal estimator using Karush-Kuhn-Tucker (KKT) conditions. Although the optimal estimator lacks a rational form, \ie, a finite-dimensional state-space realization, it can be fully described by a finite-dimensional parameter. {Leveraging this parametrization, we propose efficient algorithms that compute the optimal estimator with arbitrary fidelity in the frequency domain.} Moreover, given any finite degree, we provide an efficient convex optimization algorithm that finds the finite-dimensional state-space estimator that best approximates the optimal non-rational filter in ${\cal H}_\infty$ norm. This facilitates the practical implementation of the infinite horizon filter without having to grapple with the ill-scaled SDP from finite time. Finally, numerical simulations demonstrate the effectiveness of our approach in practical scenarios.

\end{abstract}

\section{{Introduction }} \label{sec:intro}

The Kalman filter (KF), introduced by Rudolf Kalman in 1960 \cite{kalman_new_1960}, is a fundamental tool for estimating dynamic signals generated by state-space models from noisy measurements. It has become indispensable across various fields, such as tracking \cite{chan_kalman_1979, youngrock_yoon_new_2008}, navigation \cite{grewal_application_1990,honghui_qi_direct_2002}, robotics \cite{rodriguez_kalman_1987,janabi-sharifi_kalman-filter-based_2010, chen_kalman_2012}, autonomous vehicles \cite{shoval_implementation_1997, farag_kalman-filter-based_2021}, aerospace \cite{arthur_e_bryson_kalman_1978,lefferts_kalman_1982,pittelkau_kalman_2001}, earth sciences \cite{trifu_application_2002, hargreaves_efficient_2004,galanis_applications_2006, tuan_pham_singular_1998}, biomedicine \cite{galka_solution_2004, poupon_real-time_2008, smith_kalman-based_2019}, economics and finance \cite{watson_applications_1980,schneider_analytical_1988,kellerhals_financial_2001}. Its efficacy hinges heavily on accurately modeling state-space parameters and noise statistics, which often deviate from the actual model due to statistical and approximation errors, inherent environmental uncertainties, and non-stationarities. These deviations can severely degrade performance \cite{sorenson_kalman_1985, grewal_kalman_1993, gelb_applied_2006}, posing severe risks in safety-critical applications such as aircraft navigation and autonomous vehicles. Therefore, enhancing the robustness of the Kalman filter against inaccuracies and uncertainties is crucial for ensuring safe and reliable operation.

Traditionally, robustness in the Kalman filter has been addressed by treating uncertainties as adversarial, deterministic perturbations. In this context, the $\Hinf$-filter \cite{grimble_h__1987,nagpal_filtering_1991, xie_h_1991, fu_h_1992, shaked_hsub_1992, hassibi_linear_1996, blackbook} has garnered extensive research, driven by significant advances in robust control theory \cite{blackbook, doyle_state-space_1988, zhou_robust_1996, basar_h-optimal_2008}. The $\Hinf$-filter enhances robustness by minimizing the worst-case mean-squared estimation error (MSE) attainable among all bounded energy (or power) disturbances. Although these uncertainties are presumed to arise from exogenous disturbances, the optimal $\Hinf$-filter also ensures robustness against small modeling errors in state-space parameters \cite{zames_feedback_1981}. More recently, regret-optimal filtering \cite{sabag_regret-optimal_2022, goel2023regret} has been introduced to balance performance and robustness. Unlike the $\Hinf$-filter, it minimizes the worst-case regret, defined as the excess error a causal estimator suffers compared to a clairvoyant estimator, among all bounded energy disturbances. While effective against large uncertainties, these filters neglect distributional information and may become overly conservative when faced with stochastic disturbances \cite{petersen_robust_1999}.

Distributionally robust (DR) estimation and filtering offers an alternative framework that addresses the limitations of traditional robust filtering. Pioneered by Kassam and Poor \cite{kassam_robust_1977,kassam_robust_1985} in the context of Wiener filtering \cite{wiener_extrapolation_1949}, this approach enhances robustness against uncertainties through the use of ambiguity sets of plausible statistical models. The behavior of the resulting robust filter is intricately tied to the topology of the ambiguity set, which is often constructed as a ball induced by a statistical distance or divergence. Examples include the total variation (TV) distance \cite{poor_robust_1980, vastola_robust_1984}, the Kullback-Leibler (KL) divergence \cite{levy_robust_2004, levy_robust_2013, zorzi_robustness_2017, zorzi_robust_2017}, and the Wasserstein-2 ($\Was$) distance \cite{shafieezadeh_2018, wang_robust_2021, wang_distributionally_2022-2,han_distributionally_2023, brouillon_regularization_2023, lotidis_wasserstein_2023, prabhat_optimal_2024}. The filters derived from KL-ambiguity sets have been linked \cite{levy_robust_2004, levy_robust_2013, boel_robustness_1997,hansen_robust_2005} to risk-sensitive filters \cite{jacobson_optimal_1973, speyer_optimization_1974, speyer_optimal_1992, whittle_risk-sensitive_1981, hassibi_linear_1996}, which minimize the exponentiated squared estimation error. A significant drawback of KL-ambiguity is its limited expressivity, as it only includes distributions whose support matches the nominal distribution \cite{hu_kullback-leibler_2012}. Due to its geometric interpretability as the optimal transportation metric \cite{villani_optimal_2009}, the $\Was$-distance has recently seen widespread adoption across various fields, including machine learning \cite{arjovsky_towards_2017}, computer vision \cite{liu_semantic_2020,ozaydin_omh_2024}, control \cite{DRORO, tacskesen2023distributionally, aolaritei_wasserstein_2023,brouillon2023distributionally, hajar_wasserstein_2023, kargin_wasserstein_2023}, data compression \cite{blau2019rethinking,lei2021out,malik_distributionally_2024}, and robust optimization \cite{zhao_data-driven_2018,mohajerin_esfahani_data-driven_2018,kuhn2019wasserstein,gao_distributionally_2022,yue_linear_2022,blanchet_unifying_2023,blanchet_distributionally_2024}. $\Was$-ambiguity sets offer richer expressivity, encompassing distributions with both discrete and continuous support. The $\Was$-distance also renders computationally tractable formulations for problems involving quadratic objectives, such as least mean-squared estimation \cite{nguyen_bridging_2021}, and linear-quadratic control \cite{DRORO}.

\subsection{Related Works}
Recognizing these advantages, \citet{shafieezadeh_2018} introduced a distributionally robust Kalman filter based on $\Was$-ambiguity sets confined to Gaussian distributions only. They derive state estimates at local time instances by minimizing the mean-squared error for the least favorable joint posterior distribution of the state-measurement vector, given past measurements. This iterative procedure, assuming iid Gaussian disturbances, incorporates the worst-case covariance of the previous state estimate into the nominal model for the subsequent time step. However, while this method inherently addresses state-space parameter mismatches, it lacks a global robustness guarantee over the entire time horizon and against non-iid or non-Gaussian disturbances. Similar temporally local approaches have also been studied in \cite{wang_robust_2021, wang_distributionally_2022-2,han_distributionally_2023}. More recently, \citet{lotidis_wasserstein_2023} took a different approach by imposing distributional uncertainty on the measurement noise process over the entire time horizon, assuming known iid process noise with known covariance. While the resulting filter demonstrates global robustness over the entire time horizon, the adversarial measurement noise is constrained by martingale conditions to prevent clairvoyance and dependence on future process noise realizations. Moreover, the assumption of known iid process noise is restrictive and does not provide robustness to modeling errors of the dynamics and the process noise. 
\subsection{Contributions}
In this work, we consider the Wasserstein-2 distributionally robust Kalman filtering ($\Was$-DR-KF) of linear state-space models for both finite and infinite horizons. The probability distribution of the disturbances over the entire time horizon is assumed to lie in a $\Was$-ball of a specified radius centered at a given nominal distribution. We seek the optimal causal linear estimator of a target signal that minimizes the worst-case MSE within the $\Was$-ball. We cast this as a min-max optimization problem (\cref{prob:finite horizon drf}, \cref{prob:infinite horizon drf}). Our approach differs drastically from the prior works \cite{shafieezadeh_2018, wang_robust_2021, wang_distributionally_2022-2, han_distributionally_2023, lotidis_wasserstein_2023} and possesses several advantages which can be listed as follows:

\textbf{1. Global robustness to non-iid disturbances:} In contrast to focusing on the worst-case MSE at local time instances under unknown iid disturbances \cite{shafieezadeh_2018, wang_robust_2021, wang_distributionally_2022-2, han_distributionally_2023}, our approach minimizes the cumulative MSE under the worst-case disturbance trajectory, thereby achieving global robustness for the entire horizon. Moreover, unlike \cite{lotidis_wasserstein_2023}, we impose no restrictions on the dependencies of the disturbances, accommodating arbitrarily correlated process and measurement noise sequences.

\textbf{2. Bounded steady-state error:} We derive the first infinite-horizon (aka steady-state) $\Was$-DR-KF, analogous to the steady-state Kalman and $\Hinf$-filters \cite{kailath_linear_2000,blackbook}. We show that the estimation error converges to a steady state (\cref{thm:stability}) with bounded covariance.

\textbf{3. Efficient real-time implementation:} The finite-horizon $\Was$-DR-KF requires solving an ill-scaled SDP (\cref{thm:finite horizon sdp}), rendering it impractical for real-time implementation over long time horizons. However, our infinite-horizon $\Was$-DR-KF can be implemented efficiently, thanks to our novel rational approximation, thereby overcoming the scalability issues of SDP formulation.

Our contributions are summarized as follows: 

\textbf{1. Tractable convex formulation:} We derive an SDP (\cref{thm:finite horizon sdp}) formulation for the finite-horizon problem, and a concave-convex max-min optimization problem over positive-definite Toeplitz operators (\cref{thm:dual formulation}) for the infinite horizon one. 

\textbf{2. Optimality of linear estimators for Gaussian nominal:} We focus on linear estimators while allowing the distributions in the ambiguity set to be non-Gaussian. For Gaussian nominal distributions, we show the optimality of linear estimators (\cref{thm:minimax}).

\textbf{3. Characterization of the infinite-horizon DR-KF:}  We derive the infinite-horizon $\Was$-DR-KF via KKT conditions (\cref{thm:dual formulation}) and show that the transfer function of the infinite-horizon $\Was$-DR-KF is non-rational, and thereby lacks a finite-order state-space realization. However, we also show that it can be uniquely characterized through a nonlinear finite-dimensional parametrization (\cref{thm:fixed_point}).

\textbf{4. An efficient algorithm to compute the optimal filter:} Using frequency-domain techniques, we introduce an efficient algorithm, based on the Frank-Wolfe method, to compute the optimal infinite-horizon $\Was$-DR-KF (\cref{alg:fixed_point_detailed}). We construct the best rational approximation, in the $\Hinf$-norm, of any given degree, for the non-rational optimal $\Was$-DR-KF  via a novel convex program (\cref{thm:state space filter}).

\textbf{Notations:} Bare calligraphic letters ($\K$, $\M$, etc.) are reserved for operators, with the subscripted ones ($\K_T$, $\M_T$, etc.) being finite-dimensional. $\I$ is the identity operator with a suitable block size. Asterisk $\M^\ast$ denotes the adjoint of $\M$. $\psdg$ is the usual positive-definite ordering. $\tr(\cdot)$ is the trace. $\norm{\cdot}$ is the usual Euclidean norm. $\norm[\op]{\cdot}$ and $\norm[\HS]{\cdot}$ are the $\Hinf$ (operator) and $\Htwo$ (Frobenius) norms, respectively. $\{\M\}_{+}$ and $\{\M\}_{-}$ denote the causal and strictly anti-causal parts. $\sqrt{\M}$ is the positive-definite symmetric square root. $\Sym_+^n$ is the set of psd matrices. $\abs{z}$ is the magnitude and $z^\ast$ is the conjugate of a complex number $z\!\in\!\C$. The complex unit circle is denoted by $\TT$. $\Pr$ denotes a probability distribution and $\Prob_p$ is the set of distributions with finite $p^\textrm{th}$ moment. $\E$ denotes the expectation. The Wasserstein-2 distance between distributions $\Pr_1,\Pr_2 \!\in\! \R^{n}$ is denoted by $\Was(\Pr_1,\Pr_2)$ such that
\begin{equation}\label{eq:wasserstein}
    \Was(\Pr_1,\Pr_2) \triangleq \pr{ \inf\, \E\br{ \norm{\w_1 \- \w_2}^2} }^{1/2} ,
\end{equation}
where the infimum is over all joint distributions of $(\w_1,\w_2)$ with marginals $\w_1 \!\sampled \!\Pr_1$ and $\w_2\! \sampled \!\Pr_2$.
\section{{Problem Setup} } \label{sec:prelim}
In this section, we formulate the distributionally robust filtering problem for both finite and infinite horizon settings. To this end, consider the following state-space model:
\begin{equation}\label{eq: state space}
\begin{aligned}
    x_{t+1} &= A x_{t} + B w_{t},\\
    y_{t} &= C_y x_{t} + v_{t}, \\
    s_{t} &= C_s x_{t},
\end{aligned} 
\end{equation}
At time $t\!\in\!\N$, let $x_{t} \!\in\! \R^{\d_x}$ denote the unobserved \emph{latent state}, $y_{t} \!\in\! \R^{\d_y}$ the \emph{measurement}, $s_{t} \!\in\! \R^{\d_s}$ the unobserved \emph{target signal} to be estimated, $w_{t} \!\in\! \R^{\d_w}$ the \emph{process noise}, and $v_{t} \!\in\! \R^{\d_v}$ the \emph{measurement noise}. The combined process-measurement noise sequence constitutes the \emph{exogenous disturbance}. The setup presented above is quite general and widely adopted in the estimation and filtering literature \cite{blackbook, kailath_linear_2000}. The usual state estimation problem is a specific instance of this setup with $C_s \= I$. Moreover, we assume that $(A,C_y)$ and $(A, C_s)$ are detectable and $(A, B)$ is controllable.

 We take a global view of the dynamics \eqref{eq: state space} by treating the entire signal trajectories over a fixed time horizon $T\>0$ as large column vectors. Concretely, we define the measurements vector $\y_T \!\defeq \![y_{0}; y_{1}; \dots;  y_{T\-1}] \!\in\!\R^{T\d_y}$, the target signal vector $\s_T \!\defeq \!  [s_{0}; s_{1};  \dots;  s_{T\-1}] \!\in\! \R^{Td_s}$, the process noise vector $\w_T \!\defeq \![x_{0};  w_{0};  \dots; w_{T\-2}]\!\in\! \R^{\d_x \+ (T\-1)\d_w}$, and the measurement noise vector $\vv_T \!\defeq \! [ v_{0};  v_{1}; \dots; v_{T\-1}]\!\in\! \R^{T \d_v}$. Notice that the initial state $x_0$ is considered unknown and included in the vector of process noise, $\w_T$, for convenience. With the prevailing notation, the state-space dynamics can be represented compactly as a \emph{causal linear measurement model}:
\begin{equation} \label{eq: finite horizon model}
\begin{aligned}
    \y_T &= \HH_T \w_T + \vv_T, \\
    \s_T &=  \L_T \w_T, 
\end{aligned}
\end{equation}
where $\HH_T$ and $\L_T$ are both \emph{block causal} (\ie, block lower-triangular) matrices. These matrices can be constructed easily from the state-space parameters $(A,B,C_y,C_u)$ (see \cref{app:explicit H and L}). This representation is quite general and can be extended immediately for time-varying state-space models with appropriately constructed matrices $\HH_T$ and $\L_T$.

Letting the stacked column vector $\bm{\xi}_T \!\defeq\! [ \w_T; \vv_T] \!\in\! \Xi_T$ denote the combined disturbances where $\Xi_T \!\defeq\! \R^{\d_x \+ (T\-1)\d_w + T \d_v}$, the disturbances $\bm{\xi}_T$ are distributed according to an unknown distribution $\Pr_T \!\in\! \Prob_2(\Xi_T)$. Note that the disturbances can be arbitrarily correlated in general. Wlog, we will assume $\bm{\xi}_T$ to be zero-mean for convenience. Our main assumption for $\Pr_T$ is as follows:
\begin{assumption}\label{asmp:ambiguity}
The true distribution $\Pr_T$ of disturbances $\bm{\xi}_T$ resides in a $\Was$-ball,
\begin{equation}\label{eq:wass ambiguity set} 
    \W_T(\Pr_{T}^\circ,{\rho_T}) \defeq \cl{\Pr_T \in \Prob_2(\Xi_T) \mid  \Was(\Pr_T,\, \Pr_{\circ,T}) \leq {\rho_T}},
\end{equation}
where ${\rho_T} \> 0$ is a specified radius and $\Pr_{T}^\circ \!\in\!\Prob_2(\Xi_T)$ is a given nominal disturbance distribution.
\end{assumption}

\begin{remark}\label{remark: model errors}
    Although the state-space parameters $(A,B,C_y,C_s)$ are assumed to be known perfectly, uncertainty in them can be incorporated into the disturbances without loss of generality. 
\end{remark}

\subsection{The Finite-Horizon Distributionally Robust Filtering}

A \emph{filtering policy} $\pi_T \defeq \{\pi_t \mid t\=0,\dots,T\-1\}$ is a sequence of mappings that generate estimates $\widehat{s}_t$ of $s_t$ from the past and present measurement as $\widehat{s}_t = \pi_t(y_t,y_{t-1},\dots,y_0)$. In particular, we focus on \emph{linear filtering policies} $\K_T: \y_T \mapsto \hats_T$ such that $\hats_T = \K_T \y_T$ where  $\hats_T \defeq [\widehat{s}_{0}; \widehat{s}_{1};  \dots;  \widehat{s}_{T\-1}]$ is the the column vector of estimates. We denote the class of all such policies by $\causal_T$, defined as
\begin{equation} \label{eq:causal filters}
    \causal_T \triangleq \cl{\K_T  \in \R^{T\d_s \times T\d_y} \mid \K_T \textrm{ is block lower-triangular}}.
\end{equation}
The restriction to linear filters is a common strategy in estimation literature, as general nonlinear estimators can be challenging to compute \cite{kailath_linear_2000}. Additionally, linear filters are optimal for Gaussian processes. In \cref{thm:minimax}, we establish the optimality of linear filters when the nominal is Gaussian.

For a filtering policy $\K_T$, let $\ee_T(\bm{\xi}_T,\K_T) \triangleq \hats_T - \s_T  = \T_{\K_{T}} \bm{\xi}_T$ be the estimation error where  $\T_{\K_T}: \bm{\xi}_T \mapsto \ee_T$ is the \emph{error transfer operator} defined as 
\begin{equation}\label{eq: transfer operator}
    \T_{\K_T} \triangleq \begin{bmatrix}  \K_T \HH_T - \L_T  & \K_T \end{bmatrix}.
\end{equation}
Given that the true distribution of disturbances is unknown, we focus on minimizing the \emph{worst-case mean-squared error (MSE)} across all distributions within the ambiguity set $ \W_T(\Pr_{T}^\circ,{\rho_T}) $, namely, 
\begin{equation}\label{eq:finite horizon worst case MSE}
    \wMSE_T(\K_T, \rho_T)\; \triangleq \sup_{\Pr_T \in \W_T(\Pr_{T}^\circ,\rho_T)}  \E_{\Pr_T} \br{  \norm{\ee_T(\bm{\xi}_T,\K_T)}^2}.
\end{equation}
where $\E_{\Pr_T}$ denotes the expectation under the distribution $\Pr_T$. We state the distributionally robust Kalman filtering problem for the finite-horizon setting as follows:
\begin{problem}[$\Was$-DR-KF over a finite-horizon] \label{prob:finite horizon drf}
    For a given time-horizon $T\>0$ and a radius $\rho_T\>0$, find a casual filtering policy, $\K_T\!\in\!\causal_T$, that minimizes the worst-case MSE defined in~\eqref{eq:finite horizon worst case MSE}, \ie,
    \begin{equation} \label{eq:finite horizon drf}
    \inf_{\K_T \in \causal_T} \wMSE_T(\K_T, \rho_T) \;= \inf_{\K_T \in \causal_T} \sup_{\Pr_T \in \W_T(\Pr_{T}^\circ,\rho_T)} \E_{\Pr_T} \br{  \norm{\ee_T(\bm{\xi}_T,\K_T)}^2}.
\end{equation}
\end{problem}

\begin{remark}
    The causality constraint on the estimates $\hats_T$ is crucial for filtering. Without causality enforced, \cref{prob:finite horizon drf} essentially reduces to a standard estimation problem as the nominal non-causal estimator is optimal for any $\rho_T\>0$ (\cref{lem:finite horizon non-causal}).
\end{remark}

\subsection{The Infinite-Horizon Distributionally Robust Filtering} \label{sec:infinite horizon setup}

Designing optimal filters for extended horizons can generally be impractical extended time horizons. To mitigate this, time-invariant steady-state filters are usually deployed for practical purposes. These filters can be characterized by their Markov parameters $\{\widehat{K}_{t}\}$, allowing the estimates $\{\widehat{s}_t\}$ to be computed as a convolution sum: $\widehat{s}_t \= \sum_{s=0}^{t} \widehat{K}_{t-s} y_s$. This can be expressed compactly as $\widehat{\mathbf{s}} = \mathcal{K} \mathbf{y}$, where $\mathcal{K}$ is a bounded, causal, and doubly-infinite block Toeplitz operator constructed from the Markov parameters ${\widehat{K}_{t}}$. We denote the class of all such filtering policies by $\causal$.

Here, $\y$ and $\hats$ are the doubly infinite column vectors of measurements and estimates, respectively. Furthermore, letting by $\bm{\xi}= [\w; \vv]$, and $\s$ be doubly-infinite disturbance and target signal vectors, respectively, the state-space dynamics \eqref{eq: state space} over an infinite-horizon can then be described as follows:
\begin{equation} \label{eq: infinite horizon model}
\begin{aligned}
    \y &= \HH \w + \vv, \\
    \s &=  \L \w, 
\end{aligned}
\end{equation}
where $\HH$ and $\L$ are strictly causal, doubly-infinite, block Toeplitz operators, completely described by the state-space parameters $(A,B,C_y,C_s)$. The error transfer operator $\T_\K : \bm{\xi}\mapsto \ee \defeq \hats\-\s$ under a stationary causal filtering policy $\K\in\causal$ is defined similarly as $\T_\K \defeq \begin{bmatrix}  \K \HH \- \L  & \K \end{bmatrix}$. Note that these Toeplitz operators are equivalently identified by transfer function formalism. In particular, we have $\HH \leftrightarrow H(z) \!\defeq \!C_y(zI - A)^\inv B$ and $\L \leftrightarrow L(z) \!\defeq\! C_s(zI - A)^\inv B$ for $z\in \TT$.

Instead of focusing on a fixed horizon, we consider the time-averaged steady-state worst-case MSE as the horizon approaches infinity, \ie, 
\begin{equation}\label{eq:infinite horizon worst case MSE}
    \overline{\wMSE}(\K,\rho) \triangleq \limsup_{T\to \infty}  \frac{1}{T} \wMSE_T(\K, \rho_T) = \limsup_{T\to \infty}  \frac{1}{T} \sup_{\Pr_T \in \W_T(\Pr_{T}^\circ,\rho_T)}  \E_{\Pr_T} \br{  \norm{\ee_T(\bm{\xi}_T,\K)}^2}.
\end{equation}
The limit above may generally be infinite without further specification of the asymptotics of the ambiguity set. To ensure the finiteness of the steady-state MSE, we make the following assumptions:
\begin{assumption}\label{asmp:infinite horizon assumptions}
    The nominal disturbances $\{(w_t^\circ,v_t^\circ)\}$ form a zero-mean weakly stationary random process, \ie, the cross covariance between $(w_t^\circ,v_t^\circ)$ and $(w_t^\circ,v_t^\circ)$  only depends on the difference ${t\-s}$. Furthermore, the size of the ambiguity set for horizon $T\>0$ scales as $\rho_T \sampled \rho \sqrt{T}$ for a $\rho\>0$.
\end{assumption}
The assumption on the radius $\rho_T$ for varying $T$ is justified, as the total energy of a random vector of length $T$ from a weakly stationary process scales linearly with $T$. We state the distributionally robust filtering problem for the infinite horizon as follows:
\begin{problem}[$\Was$-DR-KF over infinite-horizon] \label{prob:infinite horizon drf}
     Find a casual and time-invariant filter, $\K\!\in\!\causal$, that minimizes the steady-state worst-case MSE defined in~\eqref{eq:infinite horizon worst case MSE}, \ie,
    \begin{equation} \label{eq:infinite horizon drf}
    \inf_{\K \in \causal} \overline{\wMSE}(\K, \rho) = \inf_{\K \in \causal} \limsup_{T\to \infty}  \frac{1}{T} \sup_{\Pr_T \in \W_T(\Pr_{T}^\circ,\rho_T)}  \E_{\Pr_T} \br{  \norm{\ee_T(\bm{\xi}_T,\K)}^2}.
\end{equation}
\end{problem}

\section{Tractable Convex Formulations} \label{sec:theory}
In this section, we provide tractable formulations for the finite and infinite-horizon $\Was$-DR-KF problems. In \cref{thm:finite horizon sdp}, we present an SDP formulation for the finite-horizon problem~\ref{prob:finite horizon drf}. In \cref{thm:dual formulation}, we reduce the infinite-horizon problem~\ref{prob:infinite horizon drf} to a tractable convex program via duality. We also characterize the optimal estimator and the worst-case distribution for both settings. The proofs of the theorems presented in this section are deferred to the Appendix.

Before proceeding with the main theorems, we present a minimax theorem establishing the optimality of linear filtering policies for Gaussian nominal distributions.
\begin{theorem}[Minimax duality]\label{thm:minimax}
Let $T\>0$ be a fixed horizon and $\Pi_T$ be the class of non-linear causal estimators. Suppose that the nominal $\Pr^\circ_T$ is Gaussian. Then, the following holds:
    \begin{equation}\label{eq:minimax}
        \inf_{\pi_T \in \Pi_T} \sup_{\Pr_T \in \W_T(\Pr^\circ_T, \rho_T)} \E_{\Pr_T} \br{  \norm{\ee_T(\bm{\xi}_T,\pi_T)}^2} =  \sup_{\Pr_T \in \W_T(\Pr^\circ_T, \rho_T)} \inf_{\pi_T \in \Pi_T} \E_{\Pr_T} \br{  \norm{\ee_T(\bm{\xi}_T,\pi_T)}^2}, 
    \end{equation}
Moreover, \eqref{eq:minimax} admits a saddle point $(\pi_T^\star, \Pr_T^\star)$ such that the worst-case distribution $\Pr_T^\star$ is Gaussian and the optimal causal filter $\pi_T^\star$ is linear, \ie, $\pi_T^\star \in \causal_T$.
\end{theorem}
For simplicity and clarity, we make the following assumption for the remainder of this paper.
\begin{assumption}\label{asmp: nominal disturbance}
    The nominal disturbances are uncorrelated, \ie, 
    $\E_{\Pr^\circ_T}\br{\bm{\xi}_T \bm{\xi}_T^\ast  }= \I_T $ for any $T\>0$.
\end{assumption}

\subsection{An SDP for the Finite-Horizon Filtering}
In this section, we state the SDP formulation of \cref{prob:finite horizon drf} for a fixed horizon $T\>0$. To this end, we first state the following lemma identifying the optimal non-causal estimator.
\begin{lemma}\label{lem:finite horizon non-causal}
    Under the \cref{asmp: nominal disturbance}, $\K^\circ_T \triangleq \L_T \HH_T^\ast (\I_T \+ \HH_T \HH_T^\ast)^\inv$ is the unique, optimal, non-causal estimator minimizing the worst-case MSE in \eqref{eq:finite horizon worst case MSE} for any $\rho_T\!\geq\!0$.
\end{lemma}
This result highlights the triviality of non-causal estimation as opposed to causal estimation. In \cref{thm:finite horizon sdp}, we demonstrate that the finite-horizon $\Was$-DR-KF problem~\ref{prob:finite horizon drf} reduces to an SDP.
\begin{theorem}[An SDP formulation for finite-horizon $\Was$-DR-KF] \label{thm:finite horizon sdp}
    Let the horizon $T\>0$ be fixed and denote $\T_{\K_T^\circ}\T_{\K_T^\circ}^\ast\! \defeq \!\L_T(\I_T\+\HH_T^\ast \HH_T)^\inv \L_T^\ast$. Then, the \cref{prob:finite horizon drf} reduces to the following SDP
    \begin{equation*}\label{eq:finite horizon sdp}
        \inf_{\substack{\K_T \in \causal_T,\\ \gamma\geq 0,\, \mathcal{X}_T \in \Sym_{+}^{T\d_s} }}\!\!\gamma (\rho_T^2 \-  \Tr(\I_T)) \+ \Tr(\mathcal{X}_T )\;\; \subjto \;\;  
        \begin{bmatrix}
         \mathcal{X}_T \!&\! \gamma \I_T \!&\! 0 \\
         \gamma \I_T  \!&\! \gamma \I_T \- \T_{\K_T^\circ}\T_{\K_T^\circ}^\ast  \!&\!  \K_T \- \K_T^\circ \\
         0 \!&\! (\K_T\-\K_T^\circ)^\ast \!&\! (\I_T\+\HH_T\HH_T^\ast)^\inv
        \end{bmatrix}\! \psdgeq\!0.
    \end{equation*}
    Moreover, the worst-case disturbance $\bm{\xi}_T^\star$ can be identified from the nominal disturbances $\bm{\xi}_T^\circ$ as 
    \begin{equation}
        \bm{\xi}_T^\star = (\I_T-\gamma_{\star}^\inv \T_{\K_T^\star}^\ast \T_{\K_T^\star})^\inv \bm{\xi}_T^\circ,
    \end{equation}
    where $\gamma_\star\>0$ and $\K_T^\star$ are the optimal solutions.
\end{theorem}

\begin{remark}\label{remark:limiting_r}
As $\rho_T \to \infty$, the ambiguity set covers all bounded energy disturbances, and the optimal $\Was$-DR-KF policy, $\mathcal{K}_T^\star$, recovers the $\Hinf$-filter. Conversely, as $\rho_T \to 0$, the ambiguity set reduces to the singleton ${\Pr^\circ_T}$, and $\mathcal{K}_T^\star$ recovers the Kalman filter. Thus, adjusting $\rho_T$ allows the DR filter to interpolate between the conservative $\Hinf$-filter and the nominal Kalman filter.
\end{remark}

\tk{Notice that the variable dimension of the SDP in \cref{thm:finite horizon sdp} scales with the horizon $T$, which can be prohibitive for practical implementation for longer horizons.}

\begin{corollary}\label{thm:complexity of SDP}
    \tk{The time complexity of interior-point method for solving the SDP in \cref{thm:finite horizon sdp} with accuracy $\epsilon>0$ is $\widetilde{O}(\max(\d_y,\d_s)^6 \, T^6 \log(1/\epsilon))$.}
\end{corollary}


\subsection{A Concave-Convex Optimization for the Infinite-Horizon Filtering}
The scaling of the SDP in \cref{thm:finite horizon sdp} with the time horizon is prohibitive for many time-critical real-world applications. Therefore, we shift our focus to the infinite-horizon $\Was$-DR-KF problem~\ref{prob:infinite horizon drf} to derive the optimal steady-state filtering policy. 

Solving \cref{prob:infinite horizon drf} involves two major challenges. The first one is transforming the steady-state worst-case MSE for a fixed filtering policy $\K\!\in\!\causal$, as defined in \eqref{eq:infinite horizon worst case MSE}, to an equivalent convex optimization problem. We address this by leveraging the asymptotic convergence properties of Toeplitz matrices \cite{kargin_wasserstein_2023}. The second challenge is addressing the causality constraint on the estimator. To illustrate the triviality of non-causal estimation in the infinite-horizon setting, we present an analogous result as shown below:
\begin{lemma}\label{lem:infinite horizon non-causal}
    Under the Assumptions~\ref{asmp:infinite horizon assumptions} and~\ref{asmp: nominal disturbance}, $\K_\circ \defeq \L \HH^\ast (\I \+ \HH \HH^\ast)^\inv$ is the unique, optimal, non-causal estimator minimizing the steady-state worst-case MSE in \eqref{eq:infinite horizon worst case MSE} for any $\rho>0$.
\end{lemma}
We address the causality constraint by reformulating \cref{prob:infinite horizon drf} as a max-min optimization, where the inner minimization over the causal filtering policies is performed using the Wiener-Hopf technique \cite{wiener1931klasse,kailath_linear_2000} (see \cref{lem:wiener-hopf}). To this end, we introduce the \emph{canonical spectral factorization}\footnote{Essentially Cholesky factorization for Toeplitz operators.} $$\Delta \Delta^\ast = \I \+ \HH \HH^\ast,$$ where both $\Delta$ and its inverse $\Delta^\inv$ are causal operators. We state the equivalent formulation for the infinite-horizon $\Was$-DR-KF as follows.
\begin{theorem}[Convex formulation of infinite-horizon $\Was$-DR-KF]\label{thm:dual formulation}
    Under the Assumptions~\ref{asmp:infinite horizon assumptions} and~\ref{asmp: nominal disturbance}, the  \cref{prob:infinite horizon drf} is equivalent to the following feasible max-min problem:
    \begin{equation}
    \label{eq:infinite horizon drf convex}
        \sup_{\M \psdg 0}  \inf_{\K \in \causal} \Tr(\T_\K \T_\K^\ast \M) \quad \mathrm{s.t.} \quad \Tr(\M - 2\sqrt{\M} + \I) \leq \rho^2.
    \end{equation}
    Defining $\K_{\Htwo}\defeq \{\K_\circ \Delta\}_{\!+}\Delta^\inv$, the unique saddle point $(\K_\star,\M_\star)$ of  \eqref{eq:infinite horizon drf convex} satisfies the following:
    \begin{subequations}\label{eq:optimal K and M}
    \begin{equation}
         \K_\star = \K_{\Htwo} +  \U_\star^{-1}\cl{ \U_\star \{\K_\circ \Delta\}_{-} }_{+} \Delta^\inv,  \label{eq:optimal K from KKT} 
    \end{equation}
    \begin{equation}
         \M_\star = (\I-\gamma_\star^\inv \T_{\K_\star} \T_{\K_\star}^\ast)^{-2}, \label{eq:optimal M from KKT} 
    \end{equation}
    \end{subequations}
where $\U_\star^\ast \U_\star = \M_\star$ is the canonical spectral factorization with causal $\U_\star$ and $\U_\star^{-1}$, and $\gamma_\star>0$ is the unique value satisfying the constraint with equality, \ie, 
\begin{equation}
    \Tr\br{\pr{(\I-\gamma_\star^\inv \T_{\K_\star} \T_{\K_\star}^\ast)^\inv -\I }^2} = \rho^2.
\end{equation}
\end{theorem}
The optimal linear filter $\K_\star$, comprises the nominal Kalman (aka $\Htwo$) filter, $\K_{\Htwo}$ and an additive correction term that accounts for the correlations within the disturbance process. The correction term is derived directly from the optimal solution $\M_\star$ of \eqref{eq:infinite horizon drf convex} through spectral factorization. 

As a result of devising infinite-horizon filters achieving finite optimal value in \eqref{eq:infinite horizon drf convex}, we can deduce the boundedness of the steady-state error covariance.
\begin{corollary}\label{thm:stability}
 The steady-state error has bounded covariance under the optimal $\K_\star$ in \eqref{eq:optimal K and M}.
\end{corollary}







\section{An Efficient Algorithm} \label{sec:fixed_point}

While the standard Kalman and $\Hinf$-filters allow for finite-order spate-space realizations derived via algebraic methods, the optimal $\K_\star$ lacks such a realization since its transfer function $K_\star(z)$ is non-rational (\cref{cor:non-rational}) despite admitting a non-linear finite-dimensional parametrization (\cref{thm:fixed_point}). Thus, we adopt a novel twofold approach to develop practical DR filters: 
\setlist{leftmargin=15pt}
\begin{itemize}
\vspace{-1.5mm}
    \item[1.]  
    We introduce an efficient algorithm to compute the optimal positive-definite operator $\M_\star$ from \eqref{eq:infinite horizon drf convex}. To address the challanges posed by its infinite-dimensional nature, we use the frequency-domain representation of $\M_\star$ as the power spectral density $M_\star(z) \!\psdg\! 0$.
    \item[2.] 
    We develop a novel method to approximate the non-rational power spectral density $M_\star(z)$ in $\Hinf$-norm using positive rational functions through convex optimization. This rational approximation is then used to derive an approximate rational filter with state-space realization via \eqref{eq:optimal K and M}.
\end{itemize}
To this end, we adopt the transfer-function formalism for the rest of this paper with the correspondences: $\M \leftrightarrow M(z)$, $\U \leftrightarrow U(z)$, $\K\leftrightarrow K(z)$, and $\T_\K \leftrightarrow T_K (z)$ for $z\in\TT$. The following lemma characterizes the optimal $M_\star(z)$, implying finite-dimensional parametrization.

\begin{lemma}\label{thm:fixed_point}
Let $f:(\gamma, \Gamma) \mapsto \M $ return the unique solution of the implicit equation over $M(z)$,
\begin{equation}
        M(z) =  \gamma^2 [  \gamma I \- U(z)^\inv {\Gamma} (I \- z\overline{A})^\inv \overline{B}\, \overline{B}^\ast ( I\-z\overline{A})^{\-\ast} {\Gamma}^\ast U(z)^{\-\ast} \+ T_{K_\circ}(z)T_{K_\circ}(z)^\ast ]^{\-2}, \forall z\inn\TT
    \end{equation}
where $U(z)^\ast U(z) \= M(z)$ is the unique spectral factorization and $(\overline{A},\overline{B},\overline{C})$ are obtained from state-space parameters (see \cref{app: frequency domain details} and \eqref{eq: modified A B C}). We have that $\M_\star \= f(\gamma_\star,\Gamma_\star)$ where
\begin{equation}\label{eq:finite parameter computation}
    \Gamma_\star = \frac{1}{2\pi}\int_{-\pi}^{\pi}  U_\star(\ejw) \overline{C} (I-\ejw \overline{A})^\inv d\omega,
\end{equation}
and $\gamma_\star\>0$ is such that $\Tr(\M_\star-2\sqrt{\M_\star}+\I)=\rho^2$,
\end{lemma}
As a consequence of \cref{thm:fixed_point}, we deduce the non-rationally of the optimal $\Was$-DR-KF.
\begin{corollary}\label{cor:non-rational}
    The spectral density $M_\star(z)$ and the transfer function $K_\star(z)$ are non-rational.
\end{corollary}

\subsection{Iterative Optimization Methods in the Frequency-Domain }
Despite being a concave program, the infinite-dimensional nature of \eqref{eq:infinite horizon drf convex} hinders the direct application of standard optimization tools. To address this, we leverage frequency-domain analysis via transfer functions, enabling the use of standard tools with appropriate modifications. Specifically, we employ the modification of a Frank-Wolfe method \cite{frank_algorithm_1956,jaggi_revisiting_2013}. Our framework is versatile and can be extended to alternative approaches, including projected gradient descent \cite{goldstein1964convex}, and the fixed-point iteration method used in \cite{kargin_wasserstein_2023}. A detailed pseudocode \cref{alg:fixed_point_detailed} is provided in \cref{app: detailed pseudocode }.


\textbf{Frank-Wolfe:} We define the following function and its (Gateaux) gradient \cite{danskin}:
\begin{equation}\label{eq:wiener-hopf function}
    \Phi(\M) \triangleq \inf_{\K \in \causal} \Tr\pr{ \T_{\K} \T_{\K}^\ast \M }, \; \textrm{ and }\; \nabla \Phi(\M) \= \U^{-1} \cl{ \U \K_\circ \Delta }_{-}\cl{\U \K_\circ \Delta }_{-}^\ast \U^{-\ast} \+ \T_{\K_\circ}\T_{\K_\circ}^\ast .
\end{equation}
where $\U^\ast \U=\M$ is the spectral factorization. Rather than solving the optimization in \eqref{eq:infinite horizon drf convex} directly, the Frank-Wolfe method solves a linearized subproblem in consecutive steps. Namely, given the $k^\textrm{th}$ iterate $\M_k$, the next iterate $\M_{k+1}$ is obtained via
\begin{subequations} \label{eq: FW updates}
    \begin{equation}
    \widetilde{\M}_{k} = \argmax_{\M \psdgeq I} \,\Tr\pr{\nabla \Phi(\M_k)\,   \M } \quad \mathrm{s.t.} \quad \Tr(\M - 2\sqrt{\M} + \I) \leq \rho^2,  \label{eq:fw linear subproblem}
    \end{equation}
\begin{equation}
    \quad\quad\M_{k+1} = (1-\eta_k) \M_k + \eta_k \widetilde{\M}_k,  \label{eq:fw linear combination}
\end{equation}
\end{subequations}

where $\eta_k\inn [0,1]$ is a step-size, commonly set as $\eta_k \= \frac{2}{k+2}$ \cite{jaggi_revisiting_2013}. Letting $\G_{k} \!\defeq \!\nabla \Phi(\M_k)$ be the gradient as in \eqref{eq:wiener-hopf function}, Frank-Wolfe updates can be expressed equivalently using spectral densities as:
\begin{equation}\label{eq:frank wolfe frequency}
    \widetilde{M}_{k}(z)\= (I \- \gamma_k^\inv G_k(z))^{-2} \quad \textrm{and}\quad M_{k+1}(z) \= (1\-\eta_k) M_k(z) \+ \eta_k \widetilde{M}_k(z), \quad \forall z\in\TT
\end{equation}
where $\gamma_k\>0$ solves $\Tr\br{((I\-\gamma_k^\inv \G_k)^{-1} \- I)^2} \= \rho^2$. See \cref{app: gradients in fw} for a closed-form $G_k(z)$.

\textbf{Discretization:} Instead of the continuous domain unit circle $\TT$, we  consider its uniform discretization by $N$ points, $\TT_N \!\defeq\! \{\e^{j 2\pi n /N} \mid n\=0,\!\dots\!,N\-1\}$. While the gradient update $G_k(z)$ for frequency $z$ is applied to the next iterate $M_{k\+1}(z)$ at that frequency, calculating $G_k(z)$ requires $M_k(z^\prime)$ at all other frequencies $z^\prime \! \in \! \TT$ due to spectral factorization involved. Thus, the full update for $M_{k\+1}(z)$ needs $M_k(z^\prime)$ across the entire unit circle. This is overcome by finer discretization.

\textbf{Spectral Factorization:} Since the iterates $M_k(z)$ are non-rational spectral densities, the spectral factorization can only be performed approximately \cite{sayed_survey_2001}. Specifically, we employ the algorithm proposed in \cite{rino_factorization_1970} that uses discrete Fourier transform (DFT) and is based on Kolmogorov's method of factorization \cite{kolmogorov1939interpolation}. This method, tailored for scalar spectral densities (\ie, for scalar target signals $\d_s\=1$), proves efficient as the associated error term, featuring a multiplicative phase factor, rapidly diminishes with finer discretization $N$. Matrix-valued spectral densities can also be tackled by various other algorithms \cite{wilson_factorization_1972, ephremidze_elementary_2010}. See \cref{app: spectral factorization} for a pseudocode and details.

\textbf{Bisection: } We use bisection method to find the $\gamma_k\>0$ that solves $\Tr\br{((I\-\gamma_k^\inv \G_k)^{-1} \- I)^2} \= \rho^2$ in the Frank-Wolfe update \eqref{eq:frank wolfe frequency}. See \cref{app:bisection} for a pseudocode and further details.
\begin{remark}
    The gradient $G_k(z)$ requires computation of the finite-dimensional parameter via \eqref{eq:finite parameter computation}, which can be performed using $N$-point trapezoidal integration. See \cref{app: gradients in fw} for details.
\end{remark}
We conclude this section with the following convergence result due to \cite{jaggi_revisiting_2013, lacoste-julien_affine_2014}.
\begin{theorem}[Convergence of $\M_k$]\label{thm: convergence of FW}
    There exists constants $\delta_N\>0$, depending on discretization $N$, and $\kappa\>0$, depending only on state-space parameters \eqref{eq: state space} and $\rho$, such that the iterates in \eqref{eq: FW updates} satisfy
    \begin{equation}\label{eq:convergence rate}
        \Phi(\M_\star) - \Phi(\M_k) \leq \frac{2\kappa}{k+2}(1+\delta_N).
    \end{equation}
\end{theorem}

\subsection{Rational Approximation using \texorpdfstring{$\Hinf$}{H inf}-norm}\label{sec:RatApp}
In the preceding section, we introduced a method to compute the optimal $M_\star(z)$ approximately on the unit circle. However, the resulting filtering policy is non-rational and cannot be realized as a state-space filter. In this section, we introduce a novel technique for obtaining approximate rational filtering policies. Instead of directly approximating the filter itself, our method involves an initial step of \emph{approximating the power spectrum $M_\star(z)$} by a ratio of positive fixed order polynomials, $P(z)/Q(z)$, to minimize the $\Hinf$-norm of the approximation error. After finding a rational approximation $P(z)/Q(z)$ of $M_\star(z)$, we compute a state-space controller according to \cref{eq:optimal K and M}. For simplicity, we focus on scalar target signals, namely, $\d_s\=1$. 

Concretely, $P(z) \= \sum_{k=-m}^{m} p_k z^{-k}$ and $ Q(z) \= \sum_{k=-m}^{m} q_k z^{-k}$ are Laurent polynomials of degree $m\inn\N$ with symmetric coefficients $p_k \= p_{-k} \inn \R$ and $q_k \= q_{-k} \inn \R$. In other words, the polynomials $P(z)$ and $Q(z)$ are uniquely identified by $m+1$ real coefficients, $(p_0,\dots,p_m)$ and $(q_0,\dots,q_m)$. Given a positive spectral density $M(z)\>0$ for $z\inn \TT$, we seek positive polynomials $P(z), Q(z)\>0$ for $z\inn \TT$ of order at most $m\inn\N$ that minimize the $\Hinf$-norm of the rational approximation error, \ie,
\begin{equation}\label{eq:hinf_rational_opt}
    \min_{ \substack{p_0,\dots p_m \in \R,\\ q_0,\dots q_m \in \R,\\ \eps \geq 0 }} \; \eps \quad \subjto \quad \begin{aligned}
    &\textit{i)}\;  P(z), Q(z)>0 \textrm{ for all } z\in \TT,\\
    &\textit{ii)}\; q_0 = 1,\\
    &\textit{iii)}\; \max_{z\in\TT} \Abs{\frac{P(z)}{ Q(z)}- M(z) } \leq \eps, 
    \end{aligned}
\end{equation}
where $\eps\!\geq\!0$ denotes an \emph{upper bound on the approximation error}. The constraint $q_0\=1$ eliminates redundancy in the problem since the fraction $P(z)/ Q(z)$ is scale invariant. Unfortunately, the problem \eqref{eq:hinf_rational_opt} is not convex in all the variables. Instead, \cref{thm:convex feasibility} shows convexity for fixed $\eps\!\geq\!0$.
\begin{lemma}\label{thm:convex feasibility}
    For a fixed $\eps\!\geq\!0$, the constraints \textit{(i-iii)} define a jointly convex set for the coefficients. 
\end{lemma}
\begin{proof}
    The constraints \textit{(i-ii)} are affine inequalities, hence convex. Constraint \textit{(iii)} is equivalent to
    \begin{equation}\label{eq: convex Hinf constraint}
        P(z) -( M(z)+\eps)Q(z) \leq 0, \; \textrm{ and }\; P(z)-( M(z)-\eps)Q(z) \geq 0,  \; \textrm{ for all } z\in \TT,
    \end{equation}
    which are jointly affine inequalities in $(p_0,\dots p_m)$ and $ (q_0,\dots q_m)$, hence convex.
\end{proof}
This result enables us to obtain $m^{\textrm{th}}$-order rational approximations $P(z)/Q(z)$ of $M(z)$ with a fixed approximation precision $\eps $, signifying our tolerance for deviations from $M(z)$, by solving a \emph{convex feasibility problem}. Notice that the constraints \textit{(i)} and \textit{(iii)} (eqv. \eqref{eq: convex Hinf constraint}) involve inequalities over the entire unit circle $\TT$. Since the iterative method in \cref{alg:fixed_point_detailed} only returns the values of $M(z)$ on the discretized unit circle $\TT_N$, we can enforce these inequalities in the feasibility problem only for $\TT_N$. While being an inexact approximation for \textit{(iii)}, it is an exact characterization for \textit{(i)} as long as $N\>2m$ by the Nyquist-Shannon sampling theorem \cite{shannon1949communication}. See \cref{app: rational approximation} for a pseudocode.

Utilizing a convex feasibility oracle, our method can be used in two operational modes:
\setlist{leftmargin=15pt}
\begin{itemize}
\vspace{-1.5mm}
    \item[\textbf{1.}] \textbf{Fixed order, best precision:} By iteratively reducing the precision $\eps$ we can revise the $\eps$-feasible polynomials $P(z)$, $Q(z)$, effectively solving the non-convex problem \eqref{eq:hinf_rational_opt} to obtain the best $m^{\textrm{th}}$-order rational approximation.
    \item[\textbf{2.}]  \textbf{Fixed precision, least order:} In contrast, we can seek the lowest degree rational approximation, which achieves a fixed precision $\eps$. 
\end{itemize}

\begin{theorem}\label{thm:state space filter}
The spectral factorization $U(z)^\ast U(z) \= P(z)/Q(z)$  of a degree $m$ rational approximation $P(z)/Q(z)$ admits a rational factor $U(z)$. Furthermore, the filter obtained from $U(z)$ using \eqref{eq:optimal K and M}, \ie,  $K(z) \= K_{\Htwo}(z) \+  U(z)^{-1}\cl{ U(z) \{K_\circ(z) \Delta(z)\}_{-} }_{+} \Delta(z)^\inv$ is rational and can be realized as a state-space filter as highlighted below: 
    \begin{equation}\label{eq: state space filter}
        \begin{aligned}
            \zeta_{t+1} &=  \widetilde{F} \zeta_{t} + \widetilde{G} y_{t}, \\
            \widehat{s}_{t} &= \widetilde{H} \zeta_{t} + \widetilde{L} y_{t},
        \end{aligned}
    \end{equation}
    where $\zeta_t \!\in\! \R^{m\+ \d_x }$ is the filter state, and $(\widetilde{F},\widetilde{G},\widetilde{H},\widetilde{L})$ are determined from $(A,B,C_y,C_u)$ and $U(z)$.
\end{theorem}

\section{{Numerical Experiments}} \label{sec:simul}

In this section, we compare the performance of finite and infinite horizon DR-KF filters with $ H_2$, $ H_\infty$ filters, and other DRKFs \cite{shafieezadeh_2018}, \cite{lotidis_wasserstein_2023}. Our evaluation includes both frequency domain and time-domain analyses, highlighting the effectiveness of the rational approximation method. The nominal distribution is assumed to be Gaussian with zero mean and identity covariance. Our results demonstrate that our DR-KF (in the finite and $\infty$ horizon) provides significant advantages over other DRKFs in terms of stability, computational speed, and error reduction. The experiments were performed on a M1 Macbook Air with 8 GB of RAM.


\subsection{Frequency Domain Evaluations}\label{subsec:sim_freq}
\vspace{-3mm}
We study a typical tracking problem whose state-space model is $A=\begin{bmatrix}
    1&\Delta t\\0&1
\end{bmatrix}$, $B=\begin{bmatrix}
    0&\Delta t
\end{bmatrix}^T$,$C_y=\begin{bmatrix}
    1&0
\end{bmatrix}$, $C_s=1$
where the state corresponds to the position and velocity, the process noise is the exogenous acceleration, and $\Delta t$ is the sampling time. 
We plot the frequency response of our DR-KF using the metric $|T_K(e^{j\omega})|^2=\sigma (T_K^\ast(e^{j\omega}) T_K(e^{j\omega}))$, where $\sigma$ is the maximal singular value. We compare it to the classical $H_2$ (KF) and $H_\infty$ (robust) filters. Figure \ref{fig:fig_m_freq_11} shows that the DR-KF interpolates well between the $H_2$ (KF) and $H_\infty$ (robust) filters. Figure \ref{fig:ERR} illustrates the worst-case expected MSE. For smaller $r$, the DR-KF performs similarly to the $\mathcal{H}_2$ (KF) filter, while for larger $r$, its worst-case MSE approaches that of the robust filter. Overall, the DR-KF achieves the lowest worst-case expected MSE for any $r$. We investigate the behavior of the rational approximation across various values of the radius $r$. The results for degrees $m=1, 2, 3$ are given in Table \ref{table:worstregret} . Approximations of order greater than 2 achieve an expected MSE closely matching the non-rational DR-KF for all values of $r$.

\begin{figure}[htbp]
	\centering

	\begin{subfigure}[b]{0.45\textwidth}
		\centering
		 \includegraphics[width=0.8\textwidth]{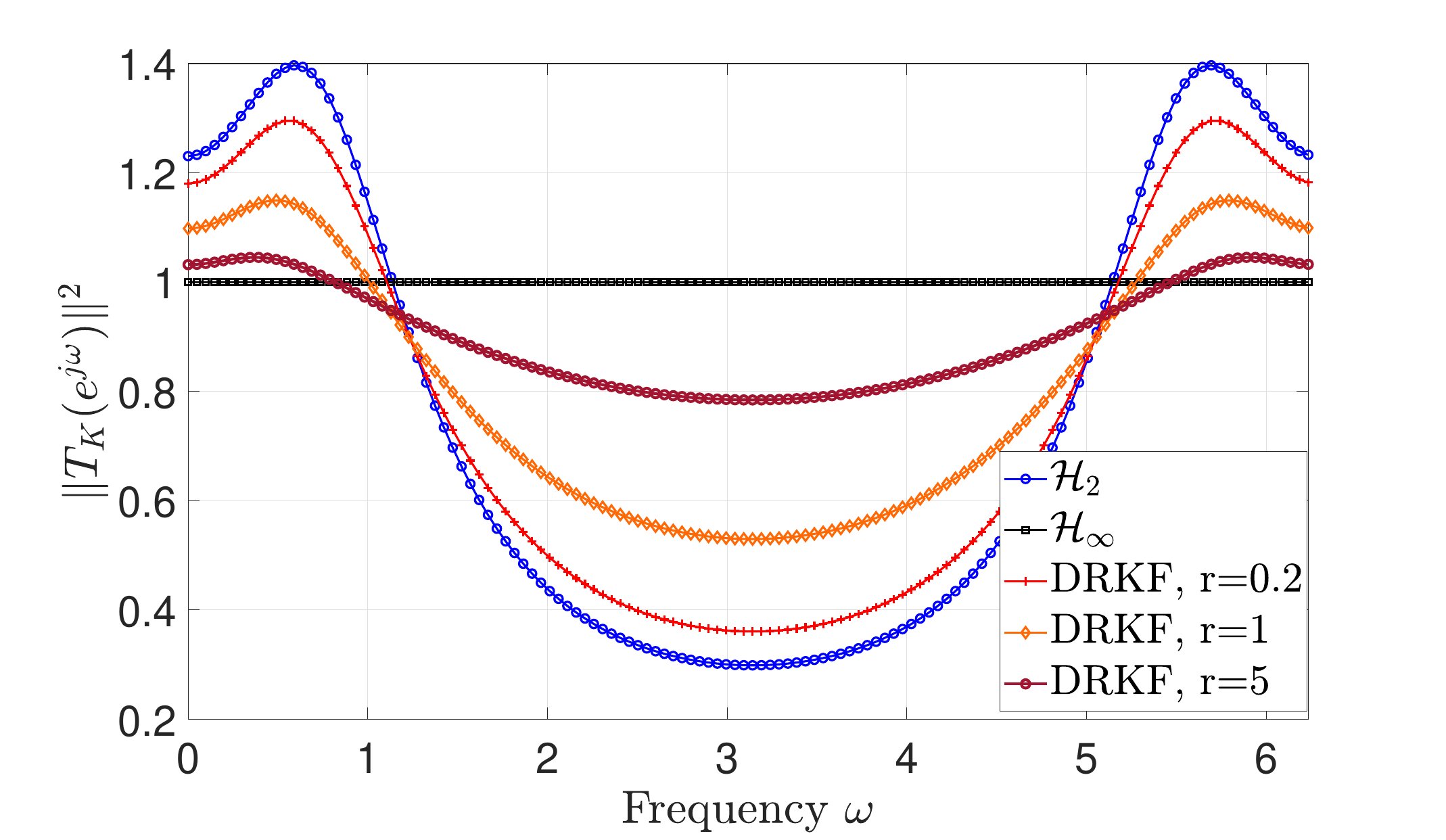}
       
		\caption{Frequency response for the tracking problem. }
		\label{fig:fig_m_freq_11}
	\end{subfigure}
	\hfill
	\begin{subfigure}[b]{0.45\textwidth}
		\centering
		\includegraphics[width=0.8\textwidth]{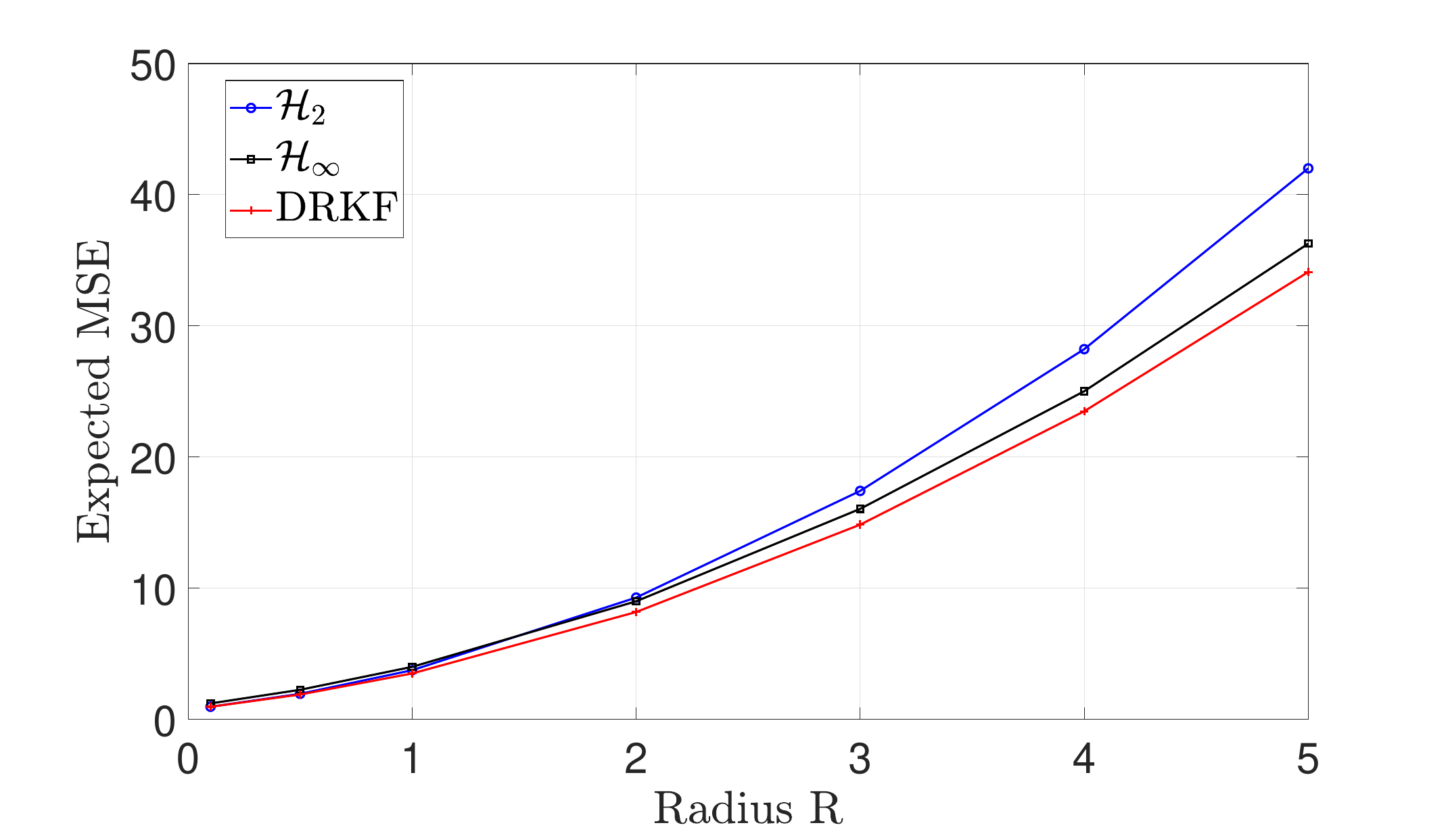}
       
		\caption{Worst-case expected MSE. }
    \label{fig:ERR}
	\end{subfigure}
       
	\caption{DR-KF  versus the $\mathcal{H}_2, \mathcal{H}_{\infty}$ filters and the variation of the expected MSE with $r$.} 
 
	\label{fig: teaser}
\end{figure}

\begin{table*}[htbp]
    \centering
    \begin{minipage}{0.5\textwidth}
        \centering
        \tiny
        \setlength\tabcolsep{4pt} 
        \begin{tabular}{|c||c|c|c|c|} 
            \hline
            \textbf{ } & \textbf{r=0.01} & \textbf{r=1} & \textbf{r=3} & \textbf{r=5}  \\
            \hline \hline
            $DRKF$ & 0.7870 &3.4948 & 14.842& 34.110  \\
            \hline
            \textbf{RA(1)} & 0.7871 & 3.5818& 15.954 & 38.327\\
            \hline
            \textbf{RA(2)} &0.7870    & 3.4948 &14.844 &34.124 \\
            \hline
            \textbf{RA(3)} & 0.7870   & 3.4948& 14.834 &34.024 \\
            \hline
        \end{tabular}
        \caption{The worst-case expected MSE of the non-rational DRKF, compared to the rational filters RA(1), RA(2), and RA(3), obtained from degree 1, 2, and 3 rational approximations to $U(e^{j\omega})$, for the system in \cref{subsec:sim_freq}.}
        \label{table:worstregret}
    \end{minipage}
    \hfill
    \begin{minipage}{0.4\textwidth}
        \centering
        \tiny
        \setlength\tabcolsep{4pt} 
            \begin{tabular}{|c||c|c|c|c|} 
        \hline
        \textbf{ } & \textbf{T=10} & \textbf{T=50} & \textbf{T=100} & \textbf{T=1000}  \\
        \hline \hline
        \textbf{DRMC} & 32.9 s & NAN & NAN & NAN  \\
        \hline
        \textbf{Our DRKF (finite)} & 0.65 s & 7.3 s & 194.9 s & NAN\\
        \hline
        \textbf{Our DRKF (infinite)} & 6.6 s & 6.6 s & 6.6 s & 6.6 s \\
        \hline
    \end{tabular}
        \caption{The running time (in seconds) of different filters for the system in section \ref{sec:stanford}. The DRMC is inefficient for T> 10, our DRKF (finite) is inefficient for T> 50 while our $\infty$ horizon DRKF can run for any horizon. }
        \label{table:times}
    \end{minipage}
\end{table*}

\vspace{-3mm}
\subsection{Time Domain Evaluations}
\vspace{-2mm}
We assess the time-domain performance of both infinite and finite horizon DR-KF filters, comparing them with $H_2$ and $H_\infty$ counterparts on the tracking problem introduced in \cref{subsec:sim_freq}. The average MSE over 50 time steps, aggregated across 1000 independent trials, is plotted. In Figure \ref{fig:figa}, under white Gaussian noise, the $H_2$ (KF) filter outperforms others. Figures \ref{fig:figb} and \ref{fig:figc} correspond to correlated Gaussian noise and the worst-case noise for the finite horizon DRKF, respectively. In Figures \ref{fig:figb} and \ref{fig:figc}, the DRKF outperforms the classical filters, and the infinite horizon DRKF matches the finite horizon one. As we increase the time horizon, solving the finite horizon SDP becomes computationally infeasible, underscoring the advantage of the infinite horizon DRKF.
    
\begin{figure}[htbp]
    \centering
        \begin{subfigure}[b]{0.33\textwidth}
        \centering
        \includegraphics[width=\textwidth]{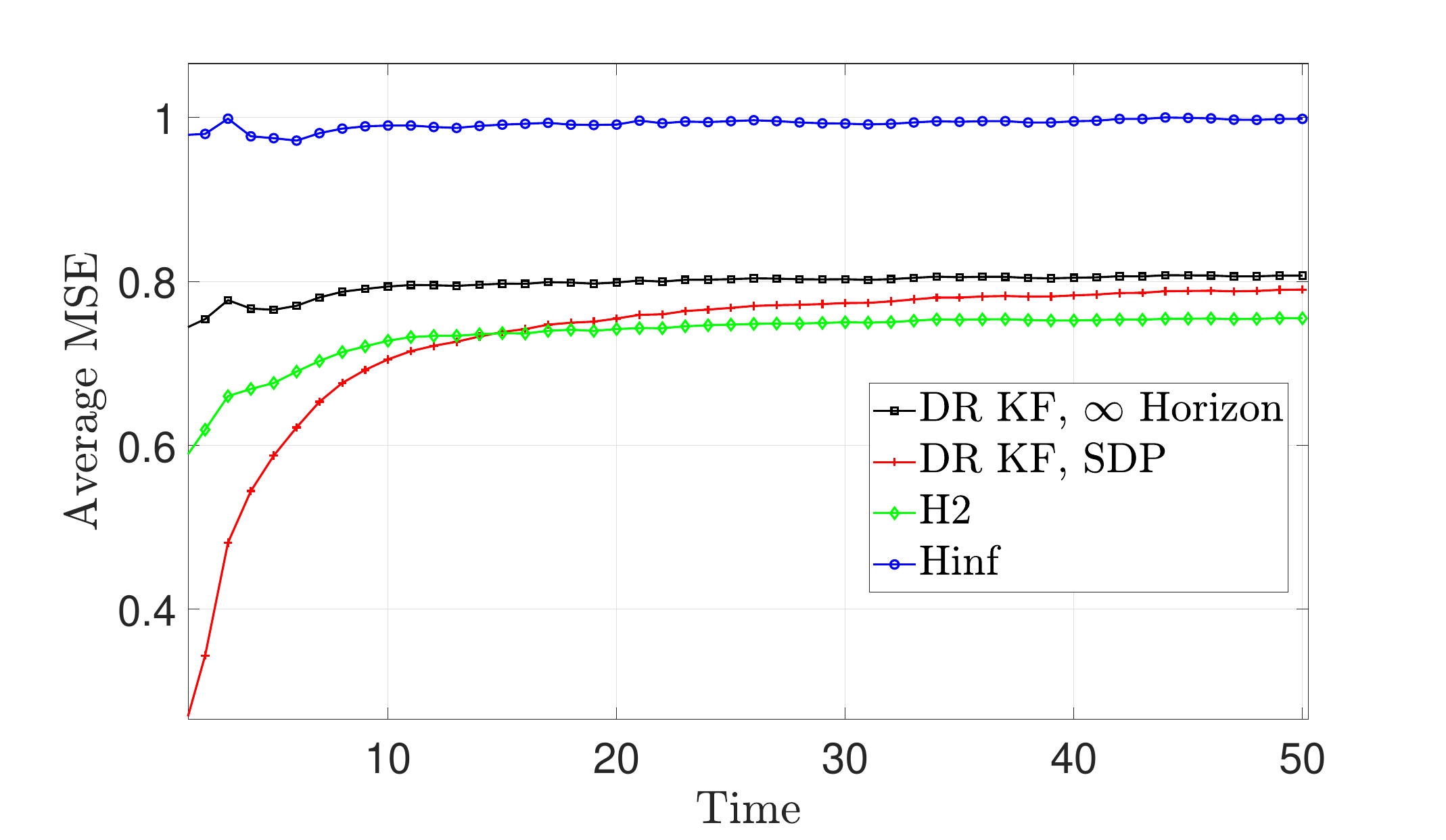}
        \caption{White noise}
        \label{fig:figa}
    \end{subfigure}
    \hfill
    \begin{subfigure}[b]{0.33\textwidth}
        \centering
        \includegraphics[width=\textwidth]{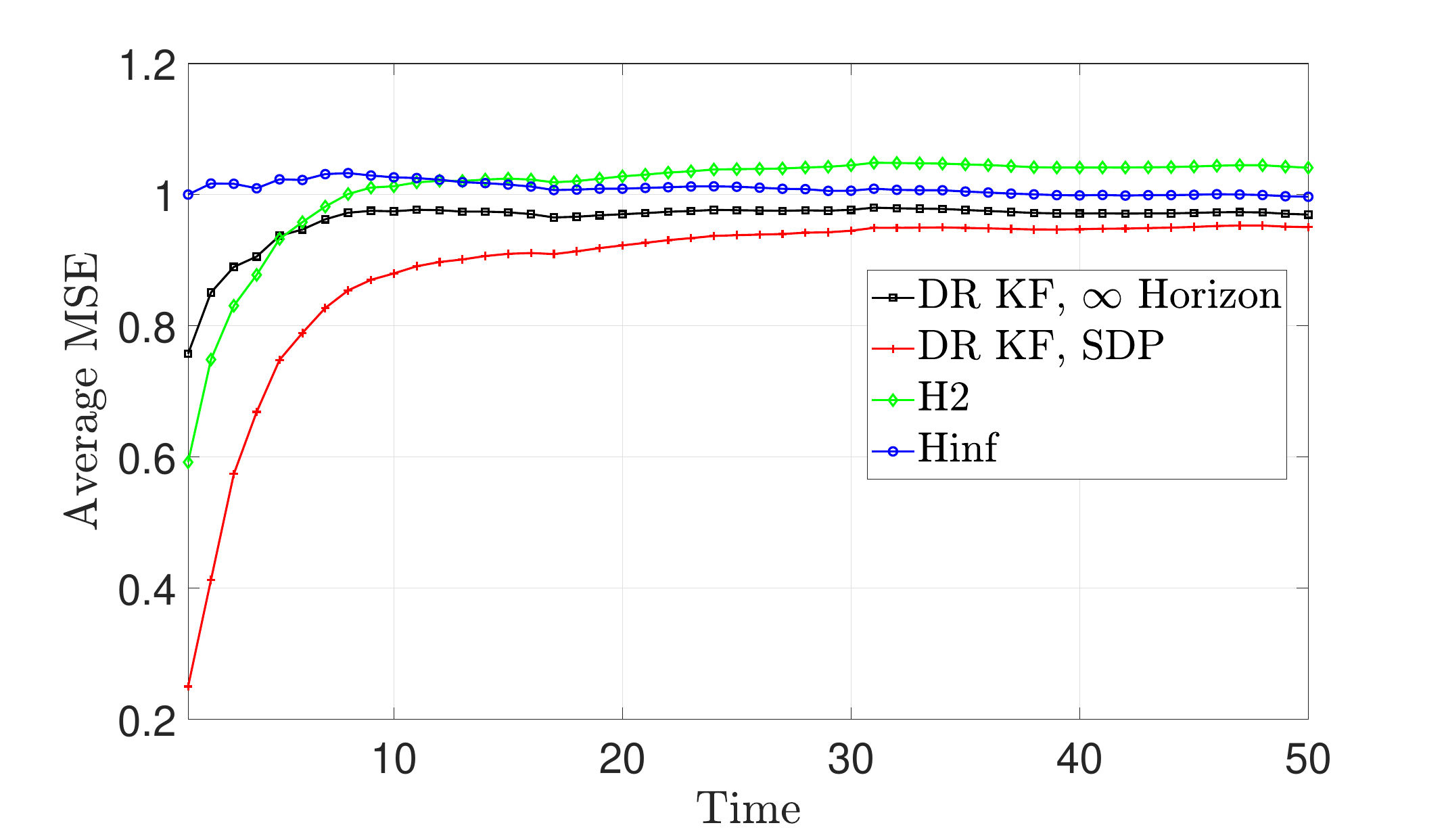}
        \caption{Correlated Gaussian noise}
        \label{fig:figb}
    \end{subfigure}
    \hfill
    \begin{subfigure}[b]{0.32\textwidth}
        \centering
        \includegraphics[width=\textwidth]{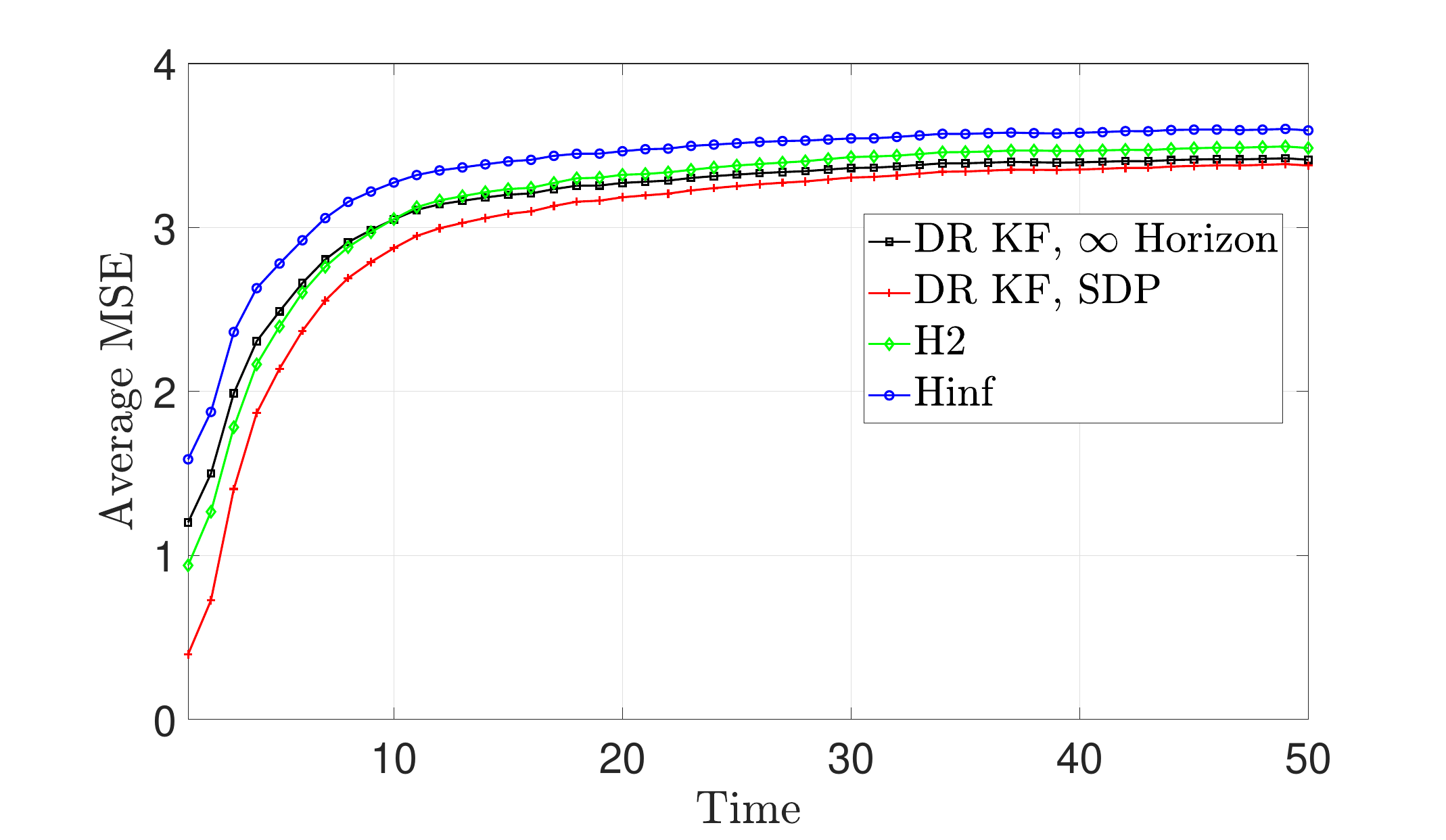}
        \caption{Worst noise for DR-KF, SDP}
        \label{fig:figc}
    \end{subfigure}

    \caption{The average MSE of the different filters for the tracking problem, under (a) white noise, (b) correlated Gaussian noise, and (c) worst-case noise for the finite horizon DR KF for the system in \cref{subsec:sim_freq}. While the $H_2$ filter (KF) performs best in (a), it behaves poorly in (b), (c). The DRKF achieves the lowest error in (b) and (c), and the finite and infinite horizon achieve similar average MSE at the end of the horizon.}
    \label{fig:time_domain}
\end{figure}

\subsection{Comparison to the DRKF in \texorpdfstring{\cite{shafieezadeh_2018}}{ Shafieezadeh-Abadeh et al. 2018}}\label{sec:epfl}
\vspace{-2mm}
We first compare against \cite{shafieezadeh_2018} which assumes the states and measurements to be in 
a Wasserstein neighborhood around a nominal at each time step, robustifying immediately against model uncertainties. Authors in \cite{shafieezadeh_2018} don't consider noise correlations across time steps, their problem setup is in the finite-horizon, and they use the Frank-Wolfe algorithm to efficiently solve the problem.
The system matrices that they consider is given by $A = \begin{bmatrix}
0.9802 & 0.0196 + 0.099\Delta \\
0 & 0.9802
\end{bmatrix}, \quad 
Q=\begin{bmatrix}
    1.9608&0.0195\\0.0195&1.9605
\end{bmatrix},\quad 
B  = \begin{bmatrix}
\sqrt Q & 0_{2\times1} 
\end{bmatrix}, \quad
C_y = \begin{bmatrix}
1 & -1
\end{bmatrix}$, $C_s=I$, and $\Delta$ represents a scalar uncertainty (taken to be 1 as in \cite{shafieezadeh_2018}).
We compare the performance of our infinite-horizon DRKF to \cite{shafieezadeh_2018} in Figure \ref{fig: teaser2} under Gaussian noise. The plot shows that our DRKF outperforms \cite{shafieezadeh_2018} and has a more stable performance, even though we are disadvantaged in two ways: 1) our filter isn't explicitely designed for model uncertainties, 2) since we only consider estimations of linear combination of the state ($C_s$ is a row vector), we get the total MSE from 2 different runs with $C_s=\begin{bmatrix}
1 & 0
\end{bmatrix}$ and $C_s=\begin{bmatrix}
0 & 1
\end{bmatrix}$, which is suboptimal.
\begin{figure}[htbp]
	\centering

	\begin{subfigure}[b]{0.45\textwidth}
		\centering
		 \includegraphics[width=0.8\textwidth]{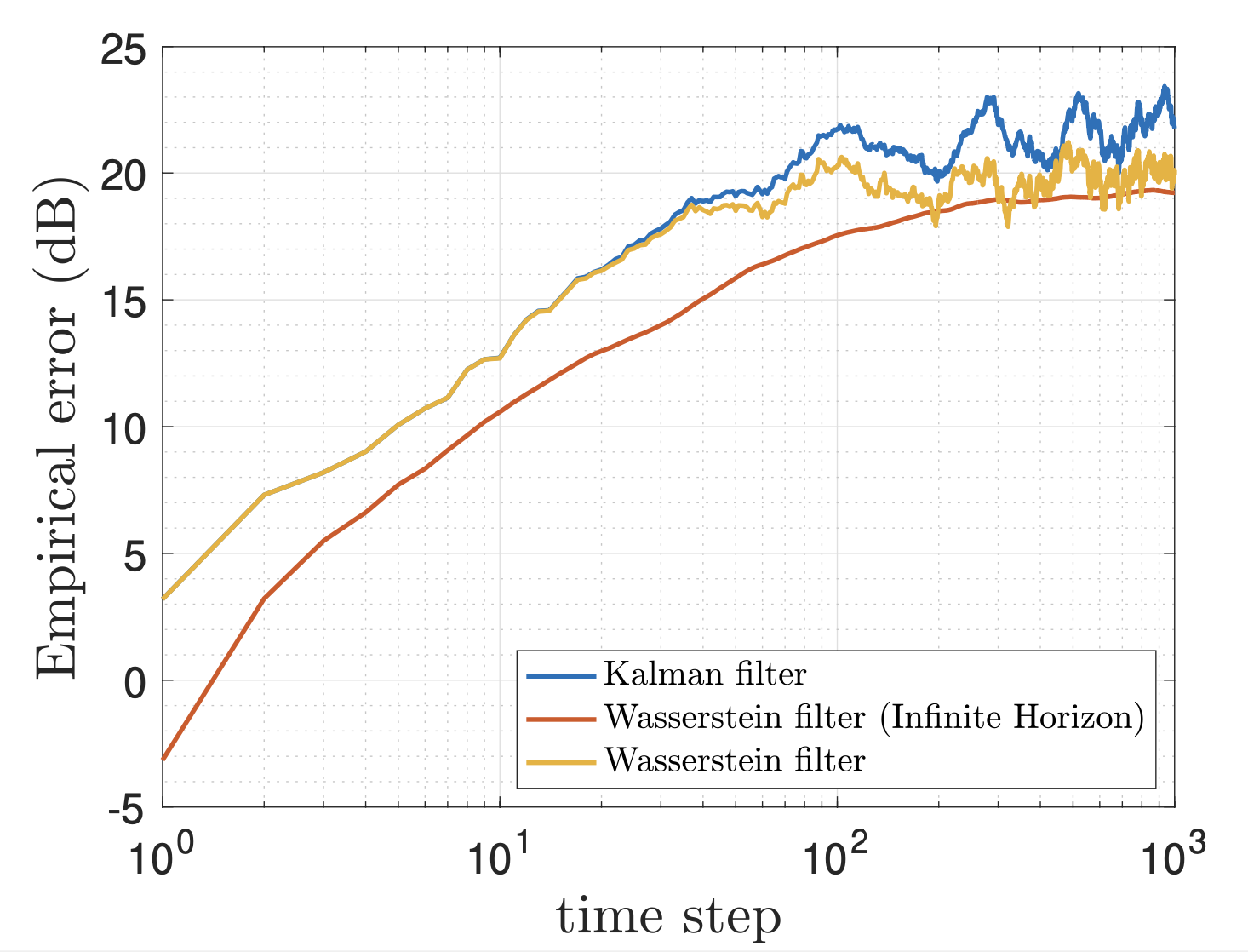}
       
		\caption{Average MSE in dB under Gaussian noise.}
		\label{fig:fig_epfl_1}
	\end{subfigure}
	\hfill
       
	\caption{Average MSE for the KF, our DRKF, and the DRKF from \cite{shafieezadeh_2018}, for system in section \ref{sec:epfl}.} 
 
	\label{fig: teaser2}
\end{figure}
\subsection{Comparison to the DR estimator of \texorpdfstring{\cite{lotidis_wasserstein_2023}}{Lotidis et al. 2023}}\label{sec:stanford}
\vspace{-3mm}
We contrast our approach with that of \cite{lotidis_wasserstein_2023}, termed linear quadratic estimator under martingale constraints (DRMC). Here are the key comparisons: 1) DRMC, akin to our approach, considers noise within a Wasserstein neighborhood around a baseline, allowing for correlations between process and measurement noise (achieved through a martingale sequence constraint). 2) DRMC assumes the process noise is sampled from the baseline and doesn't lie in the Wasserstein ball, a more restrictive assumption compared to ours. 3) DRMC's problem formulation is in the finite-horizon, claiming to have an efficient converging method to solve it. With a horizon of $T=10$, they test their approach on a simple 1D system ($A=B=C_y=C_s=1$) , which we also use for comparison. For $r=0.2 \sqrt T$ and under the worst-case noise for our finite-horizon DRKF, the average MSE for DRMC is 0.86, closely matching our finite-horizon DRKF at 0.86 and our infinite-horizon DRKF at 0.88 at $T=10$.  
For the same $r=$ under the worst-case noise of DRMC, the average MSE for DRMC is 0.78, close to our DRKF (0.81 for the finite-horizon and 0.83 for the infinite-horizon at $T=10$). This shows that using the infinite-horizon controller for short horizons does not significantly compromise performance. Similar results are observed for other values of $r$.
While the performances in this simple example are comparable, our filter is anticipated to excel for higher-diemsnional systems, due to its explicit consideration of robustness over process noise. However, our DRKFs outshine DRMC in efficiency. DRMC takes 32.9 seconds for $T=10$, and becomes computationally infeasible beyond that. Our finite-horizon DRKF is faster and efficient up to $T=50$, and our infinite-horizon DRKF remains unaffected by the time horizon. For details, see Table \ref{table:times}.


\section{{Conclusion}} \label{sec:conc}
\tk{The main limitation in our work is that our $H_\infty$-rational approximation method is limited to scalar target signals (\ie, $C_s$ is a row vector). Future work will address this limitation.}
\newpage

\bibliography{refs}
\bibliographystyle{unsrtnat}

\newpage
\appendix
\onecolumn
\begin{center}
{\huge Appendix}
\end{center}


\section{{Simulation Results}} \label{ap:sim}

\subsection{Another tracking problem}
We study another tracking problem, standard in the filtering community, whose state-space model is 

$F=\begin{bmatrix}
    1&0&0&0\\\Delta t &1&0&0\\0&0&1&0\\0&0&\Delta t&1
\end{bmatrix}$, $G=\begin{bmatrix}
    \Delta t &0\\0.5 (\Delta t)^2&0\\0&\Delta t\\0&0.5 (\Delta t)^2
\end{bmatrix}$,$H=\begin{bmatrix}
    0&1&0&0\\0&0&0&1
\end{bmatrix}$, $L=\begin{bmatrix}
    0&0&0&1
\end{bmatrix}$, 

with $\Delta t=1$.
The results are shown in the plots below.
\begin{figure}[htbp]
	\centering

		\centering
		 \includegraphics[width=0.4\textwidth]{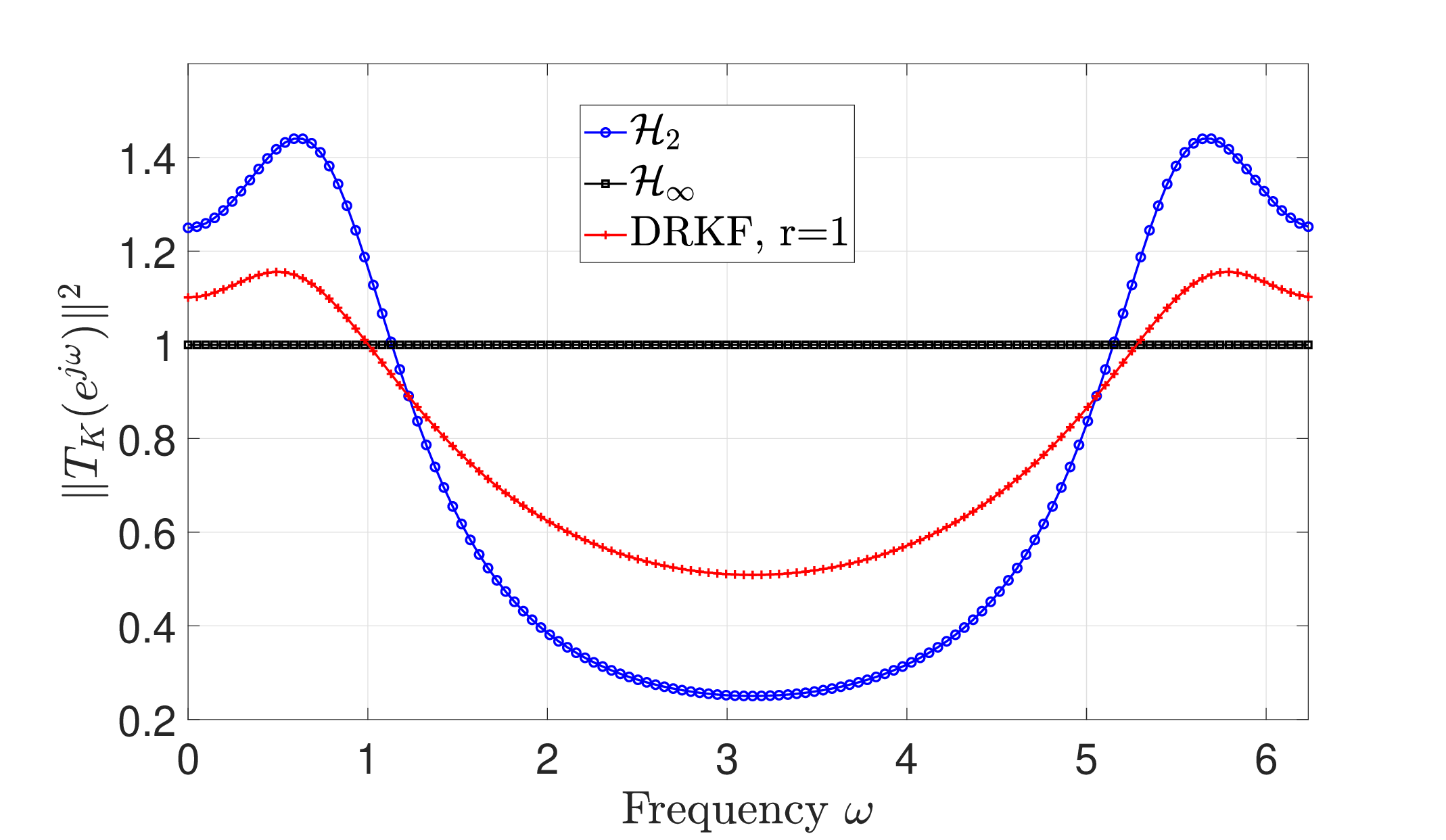}
       
		\caption{The frequency response of different filters ($\mathcal{H}_2, \mathcal{H}_{\infty}$ and DRKF) for the tracking problem in section \ref{ap:sim}. The worst-case expected MSE is 3.99 for $H_\infty$ , 3.77 for $H_2$ and 3.47 (lowest) for DRKF. }
		\label{fig:fig_m_freq_12}

\end{figure}

\begin{figure}[ht]
    \centering
        \begin{subfigure}[b]{0.45\textwidth}
        \centering
        \includegraphics[width=\textwidth]{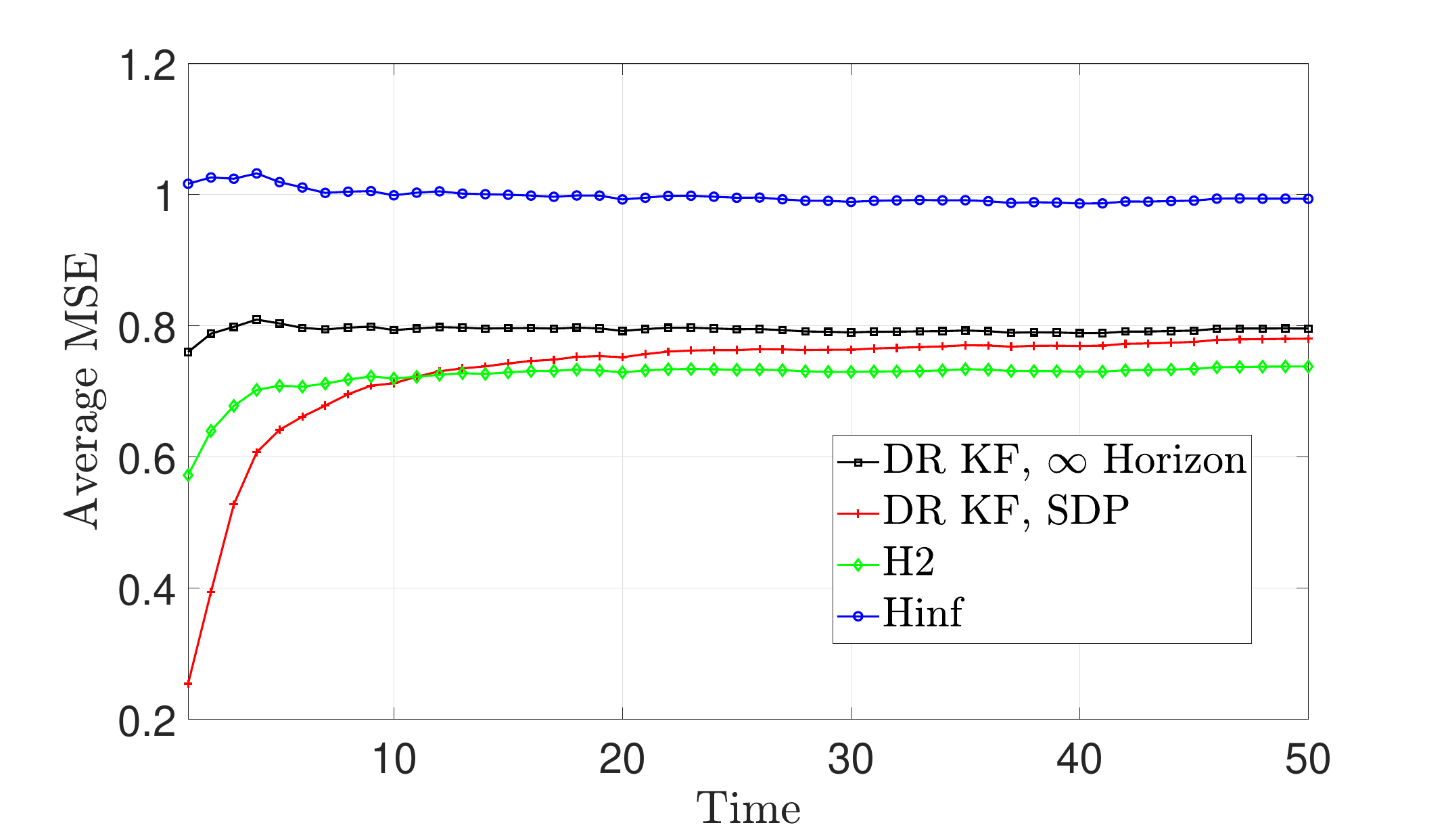}
        \caption{White noise}
        \label{fig:figa1}
    \end{subfigure}
    \hfill
    \begin{subfigure}[b]{0.45\textwidth}
        \centering
        \includegraphics[width=\textwidth]{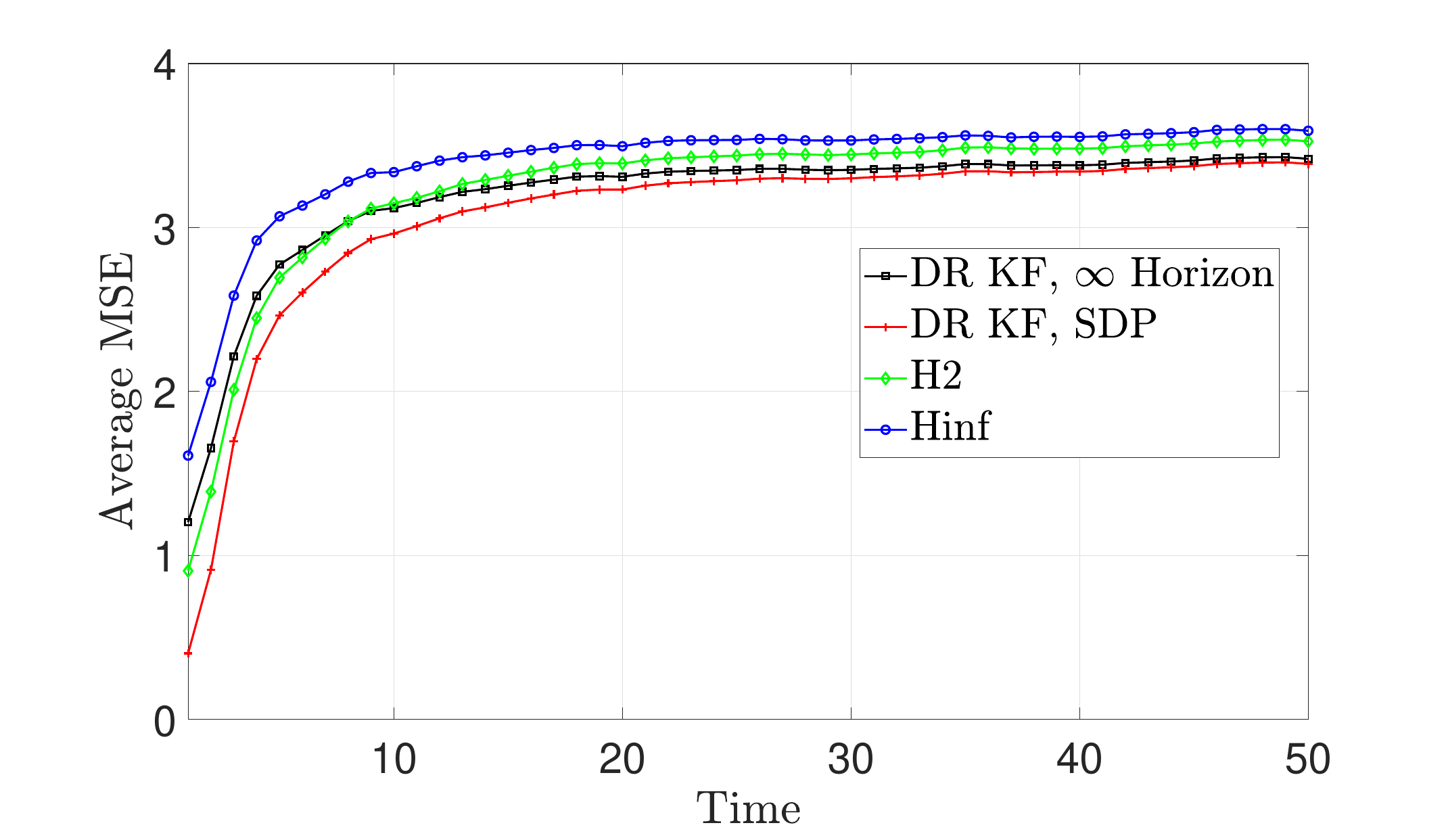}
        \caption{Worst noise for DR-KF, SDP}
        \label{fig:figc1}
    \end{subfigure}

    \caption{The average MSE of the different filters horizon under different disturbances for the tracking problem in section \ref{ap:sim}. (a) is white noise, while (b) is the worst-case noise for the finite horizon DR KF (SDP). While the KF performs best under gaussian noise, the DRKF achieves the lowest error in most of other scenarios (including the more realistic case of correlated noise), and the finite and infinite horizon achieve similar avergae MSE at the end of the horizon.}
    \label{fig:time_domain2}
\end{figure}


\section{Additional Discussion on the Problem Setup} \label{ap:problem setup}

\subsection{Explicit Form of Finite-Horizon Model in \texorpdfstring{\eqref{eq: finite horizon model}}{(3)}}\label{app:explicit H and L}
The causal linear measurement model for the finite-horizon case in \eqref{eq: finite horizon model} can be stated explicitly as follows:
\begin{subequations}
    \begin{equation}
    \underbrace{\begin{bmatrix}
        y_{0} \\ y_{1} \\ y_{2} \\ \vdots  \\ y_{T}
    \end{bmatrix}}_{\y} = \underbrace{\begin{bmatrix}
        C_y    & 0   & 0  & \dots & 0  \\
        C_y A   & C_y B  & 0  & \dots & 0  \\
        C_y A^2 & C_y AB & C_y B & \dots & 0  \\
        \vdots & \vdots & \vdots & \ddots & \vdots \\
        C_y A^{T} & C_y A^{T\-1}B & C_y A^{T\-2}B & \ddots & C_y B  
    \end{bmatrix}}_{\HH_{T}}
    \underbrace{\begin{bmatrix}
         x_{0} \\ w_{0} \\  w_{1} \\ \vdots  \\ w_{T\-1}
    \end{bmatrix}}_{\w} + \underbrace{\begin{bmatrix}
         v_{0} \\ v_{0} \\  v_{1} \\ \vdots  \\ v_{T}
    \end{bmatrix}}_{\vv}
\end{equation}
\begin{equation}
    \underbrace{\begin{bmatrix}
        s_{0} \\ s_{1} \\ s_{2} \\ \vdots  \\ s_{T}
    \end{bmatrix}}_{\s} = \underbrace{\begin{bmatrix}
        C_s    & 0   & 0  & \dots & 0  \\
        C_s A   & C_s B  & 0  & \dots & 0  \\
        C_s A^2 & C_s AB & C_s B & \dots & 0  \\
        \vdots & \vdots & \vdots & \ddots & \vdots \\
        C_s A^{T} & C_s A^{T\-1}B & C_s A^{T\-2}B & \ddots & C_s B  
    \end{bmatrix}}_{\L_{T}}
    \underbrace{\begin{bmatrix}
         x_{0} \\ w_{0} \\  w_{1} \\ \vdots  \\ w_{T\-1}
    \end{bmatrix}}_{\w}
\end{equation}
\end{subequations}

A similar construction of $\HH_T$ and $\L_T$ for time-varying systems can be performed by replacing the causal block elements of $\HH_T$ and $\L_T$ with appropriate coefficients derived from the time-varying dynamics.

\section{Proofs of Theorems Related to Finite-Horizon Filtering}\label{app: finite horizon proofs}

\subsection{Proof of \texorpdfstring{\cref{thm:minimax}}{Theorem 3.1}}
Before we proceed with the proof, we first state the following useful deifnitions and results.

\begin{definition}[{Bures-Wasserstein distance \cite{bhatia_bures-wasserstein_2017}}] \label{def:bw distance}
For any two psd matrices $\Sigma_1,\Sigma_2 \in \Sym_{+}^{d}$, the Bures-Wasserstein distance between them is defined as follows:
    \begin{equation}\label{eq:bw_dist}
     \BW(\Sigma_1, \Sigma_2) \triangleq \sqrt{\Tr\br{ \Sigma_1 + \Sigma_2 - 2 \pr{\Sigma_1^{1/2} \Sigma_2 \Sigma_1^{1/2}}^{1/2} }  }.
\end{equation}
\end{definition}

\begin{definition}[{Gelbrich distance \cite{gelbrich}}]\label{def:gelbrich}
   For any two distributions $\Pr_1, \Pr_2 \in \Prob(\R^d)$ with means $\mu_1,\mu_2 \in \R^d$ and covariances $\Sigma_1,\Sigma_2 \in \Sym_{+}^{d}$, respectively, the Gelbrich distance between them is defined as follows: 
\begin{equation}\label{eq:gelbrich}
    \mathsf{G}(\Pr_1, \Pr_2) \triangleq \sqrt{\norm{\mu_1  -\mu_2}^2 +   \BW(\Sigma_1, \Sigma_2)^2 }.
\end{equation}
\end{definition}

\begin{lemma}[{Gelbrich bound \cite[Thm. 2.1]{gelbrich}}] \label{lem:gelbrich bound}
Consider two distributions $\Pr_1, \Pr_2 \in \Prob(\R^d)$ with means $\mu_1,\mu_2 \in \R^d$ and covariances $\Sigma_1,\Sigma_2 \in \Sym_{+}^{d}$, respectively. The $\Was$-distance between them satisfies
\begin{equation}\label{eq:gelbrich bound}
    \Was(\Pr_1, \Pr_2) \geq \mathsf{G}(\Pr_1, \Pr_2),
\end{equation}
where equality is attained if both $\Pr_1$ and $\Pr_2$ are Gaussian distributions.
\end{lemma}

\begin{lemma}[{Causal MMSE of Gaussian \cite{kailath_linear_2000}}] \label{lem:gaussian estimation}
Suppose the disturbances are distributed as Gaussian, \ie, $\bm{\xi}_T \sampled \normal(\bm{\mu}_T, \bm{\Sigma}_{T})$ with mean $\bm{\mu}_T \in \Xi_T$ and covariance $\bm{\Sigma}_{T}$. Consider causal mean-square estimation of $\s_T$ from $\y_t$, \ie, 
    \begin{equation}
        \inf_{\pi_T \in \Pi_T} \E \br{  \norm{\ee_T(\bm{\xi}_T,\pi_T)}^2}.
    \end{equation}
    Then, there exists a causal (block lower-diagonal) matrix $\K_T^\star$ and a vector $\mathbf{b}_T^\star$, such that the optimal causal estimator $\pi_T^\star:\y_T \mapsto \hats_T$ is affine with the following form:
    \begin{equation}
        \hats_T = \K_T^\star \y_t + \mathbf{b}_T^\star.
    \end{equation}
\end{lemma}

\paragraph{\textit{Proof of \cref{thm:minimax}:} }
Clearly, we have the following weak duality,
\begin{equation}\label{eq: weak duality}
    \sup_{\Pr_T \in \W_T(\Pr^\circ_T, \rho_T)} \inf_{\pi_T \in \Pi_T} \E_{\Pr_T} \br{  \norm{\ee_T(\bm{\xi}_T,\pi_T)}^2} \leq  \inf_{\pi_T \in \Pi_T} \sup_{\Pr_T \in \W_T(\Pr^\circ_T, \rho_T)} \E_{\Pr_T} \br{  \norm{\ee_T(\bm{\xi}_T,\pi_T)}^2}.
\end{equation}
Let $\bm{\Sigma}_T^\circ \psdg 0$ be the covariance of the nominal distribution $\Pr_T^\circ$. We start by bounding the lhs of \eqref{eq: weak duality} as follows:
\begin{align}
     \sup_{\Pr_T \in \W_T(\Pr^\circ_T, \rho_T)} \inf_{\pi_T \in \Pi_T} \E_{\Pr_T}& \br{  \norm{\ee_T(\bm{\xi}_T,\pi_T)}^2} \stackrel{(a)}{\leq} \sup_{\mathsf{G}(\Pr_T, \Pr_T^\circ) \leq \rho_T} \inf_{\pi_T \in \Pi_T} \E_{\Pr_T} \br{  \norm{\ee_T(\bm{\xi}_T,\pi_T)}^2}, \\
     &\stackrel{(b)}{\leq} \sup_{\mathsf{G}(\Pr_T, \Pr_T^\circ) \leq \rho_T} \inf_{\K_T \in \causal_T} \E_{\Pr_T} \br{  \norm{\ee_T(\bm{\xi}_T,\K_T)}^2}, \\
     &= \sup_{\mathsf{G}(\Pr_T, \Pr_T^\circ) \leq \rho_T} \inf_{\K_T \in \causal_T} \E_{\Pr_T} \br{  \bm{\xi}_T^\ast \T_{\K_T}^\ast  \T_{\K_T}\bm{\xi}_T },\\
     &\stackrel{(c)}{=} \sup_{\mathsf{G}(\Pr_T, \Pr_T^\circ) \leq \rho_T} \inf_{\K_T \in \causal_T}   \Tr\pr{ \T_{\K_T}^\ast  \T_{\K_T}  \E_{\Pr_T} \br{  \bm{\xi}_T \bm{\xi}_T^\ast}}, \\
     &\stackrel{(d)}{=} \sup_{\BW(\bm{\Sigma}_T,\bm{\Sigma}_T^\circ) \leq \rho_T} \inf_{\K_T \in \causal_T}   \Tr\pr{ \T_{\K_T}^\ast  \T_{\K_T}  \bm{\Sigma}_T},\label{eq: first bound on the minmax}
\end{align}
where $(a)$ follows from the Gelbrich bound (\cref{lem:gelbrich bound}), $(b)$ follows from $\causal_T \subset \Pi_T$, $(c)$ follows from linearity of cyclic property of trace and the linearity of trace and the expectation, $(d)$ follows from the definition of the Gelbrich distance (\cref{def:gelbrich}). Note that, we can in general take the distributions involved to be zero-mean since any non-zero mean can be incorporated as an additive constant to the estimator, canceling the mean. Therefore, without loss of generality, we can restrict ourselves to zero-mean disturbances and linear estimators (instead of affine). 

Following a similar reasoning, we obtain the following upper bound on the rhs of \eqref{eq: weak duality},
\begin{align}
    \inf_{\pi_T \in \Pi_T}  \sup_{\Pr_T \in \W_T(\Pr^\circ_T, \rho_T)} \E_{\Pr_T} \br{  \norm{\ee_T(\bm{\xi}_T,\pi_T)}^2} &\leq  \inf_{\K_T \in \causal_T} \sup_{\BW(\bm{\Sigma}_T,\bm{\Sigma}_T^\circ) \leq \rho_T}  \Tr\pr{ \T_{\K_T}^\ast  \T_{\K_T}  \bm{\Sigma}_T}.\label{eq: minmax bw }
\end{align}
Notice that the objective in the right-hand side of \eqref{eq: minmax bw } is affine in $\bm{\Sigma}_T$ (hence concave) and quadratic in $\K_T$ (hence strictly convex whenever $\bm{\Sigma}_T \psdg 0$). Furthermore, the constraint set $\causal_T$ is affine, and the constraint $\BW(\bm\Sigma_T,\bm\Sigma_T^\circ)$ is convex \cite{bhatia_bures-wasserstein_2017}. Therefore, we have the following minimax duality.
\begin{equation}\label{eq: minimax with bw dist}
    \inf_{\K_T \in \causal_T} \sup_{\BW(\bm{\Sigma}_T,\bm{\Sigma}_T^\circ) \leq \rho_T}  \Tr\pr{ \T_{\K_T}^\ast  \T_{\K_T}  \bm{\Sigma}_T} =  \sup_{\BW(\bm{\Sigma}_T,\bm{\Sigma}_T^\circ) \leq \rho_T} \inf_{\K_T \in \causal_T}   \Tr\pr{ \T_{\K_T}^\ast  \T_{\K_T}  \bm{\Sigma}_T}.
\end{equation}
We denote the saddle point of \eqref{eq: minimax with bw dist} by $(\K_T^\star, \bm\Sigma_T^\star)$. Notice that, when the nominal distribution is Gaussian $\Pr_T^\circ \defeq \normal(0,\bm{\Sigma}_T^\circ)$, the Gaussian distribution $\Pr_{T}^\star \defeq \normal(0, \bm\Sigma_T^\star)$ and the causal estimator $\pi_T^\star \defeq \K_T^\star$ achieve the upper bound \eqref{eq: first bound on the minmax} with equality. Thus, from \eqref{eq: weak duality} and \eqref{eq: minmax bw }, we obtain the desired result.
\qed

\subsection{Proof of \texorpdfstring{\cref{lem:finite horizon non-causal}}{Lemma 3.3}}\label{app:proof of non-causal finite horizon}
Before we proceed with the proof, we state the following result, which is the backbone for both \cref{thm:finite horizon sdp} and \cref{thm:dual formulation}.

\begin{theorem}[Strong Duality in the Finite Horizon]\label{thm: strong duality finite horizon}
     Let the horizon $T\>0$ be fixed and $\K_T$ be a given estimator, which can be non-causal in general. Under the \cref{asmp: nominal disturbance}, the finite-horizon worst-case MSE \eqref{eq:finite horizon worst case MSE} suffered by $\K_T$, \ie,
    \begin{equation}\label{eq: alternative finite horizon worst case MSE}
        \wMSE_T(\K_T, \rho_T)\; = \sup_{\Pr_T \in \W_T(\Pr_{T}^\circ,\rho_T)}  \E_{\Pr_T} \br{  \bm{\xi}_T^\ast \T_{\K_T}^\ast \T_{\K_T} \bm{\xi}_T },
    \end{equation}
    attains a finite value and is equivalent to the following dual problem:
    \begin{equation}\label{eq: stong duality finite horizon}
     \wMSE_T(\K_T, \rho_T)\; = \inf_{\substack{\gamma \geq 0 } }  \gamma \rho_T^2 +  \gamma\Tr\br{ ( \I_T - \gamma^{\-1}\T_{\K_T} \T_{\K_T}^\ast )^\inv - \I_T } \quad  \textrm{s.t.} \quad  \gamma \I_T \psdg \T_{\K_T} \T_{\K_T}^\ast.
\end{equation}
Furthermore, the worst-case disturbance, $ \bm{\xi}_T^\star$, can be identified from the nominal disturbance, $ \bm{\xi}_T^\circ$, as 
\begin{equation}\label{eq:worst case distrubance finite horizon}
    \bm{\xi}_T^\star = ( \I_T - \gamma_\star^\inv \T_{\K_T}^\ast \T_{\K_T})^{\-1}  \bm{\xi}_T^{\circ}, 
\end{equation}
where $\gamma_\star$ is the optimal solution of \eqref{eq: stong duality finite horizon} and solves the following equation uniquely:
\begin{equation}\label{eq:worst gamma strong duality finite horizon}
        \Tr\br{( (\I_T - \gamma_\star^\inv \T_{\K_T} \T_{\K_T}^\ast )^\inv - \I_T)^2} = \rho_T^2.
\end{equation}  
\end{theorem}
\begin{proof}
    The proof follows closely from \cite[Thm. 2 \& 3]{DRORO} (and also from \cite[Thm. IV.1]{hajar_wasserstein_2023}) by replacing the matrix $C$ in Thm.2 of \cite{DRORO} with $\T_{\K_T}^\ast \T_{\K_T}$. In that case, we get that 
    \begin{equation}\label{eq: stong dual without interchanging T_K}
     \wMSE_T(\K_T, \rho_T)\; = \inf_{\substack{\gamma \geq 0 } }  \gamma \rho_T^2 +  \gamma\Tr\br{ ( \I_T - \gamma^{\-1}\T_{\K_T}^\ast \T_{\K_T} )^\inv - \I_T } \quad  \textrm{s.t.} \quad  \gamma \I_T \psdg \T_{\K_T}^\ast \T_{\K_T},
\end{equation}
and the characterization of the optimal solution $\gamma_\star\geq 0$ as
\begin{equation}\label{eq:worst_gamma  without interchanging T_K}
        \Tr\br{( (\I_T - \gamma_\star^\inv \T_{\K_T}^\ast \T_{\K_T} )^\inv - \I_T)^2} = \rho_T^2.
\end{equation}  
Notice that \eqref{eq: stong dual without interchanging T_K} and \eqref{eq:worst_gamma  without interchanging T_K}  involve the term $ \T_{\K_T}^\ast \T_{\K_T}$ whereas the desired formulations in \eqref{eq: stong duality finite horizon} and \eqref{eq:worst gamma strong duality finite horizon}  involve the term $\T_{\K_T} \T_{\K_T}^\ast $. To obtain the desired formulations, we appeal to matrix inversion identity, \ie,
\begin{align}
    ( \I_T - \gamma^{\-1}\T_{\K_T}^\ast \T_{\K_T} )^\inv &= \I_T + \T_{\K_T}^\ast ( \gamma \I_T - \T_{\K_T} \T_{\K_T}^\ast )^\inv \T_{\K_T},\\
    &= \I_T + \gamma^\inv \T_{\K_T}^\ast (  \I_T - \gamma^\inv \T_{\K_T} \T_{\K_T}^\ast )^\inv \T_{\K_T}, \label{eq: matrix inversion lemma}
\end{align}
where the exact block dimensions of the identity operator $\I_T$ differ depending on where they appear and should be inferred from the context. We can evaluate the trace in \eqref{eq: stong dual without interchanging T_K} involving $ \T_{\K_T}^\ast \T_{\K_T}$ as
\begin{align}
    \Tr\br{ ( \I_T - \gamma^{\-1}\T_{\K_T}^\ast \T_{\K_T} )^\inv - \I_T } &=  \Tr\br{ \I_T + \gamma^\inv \T_{\K_T}^\ast (  \I_T - \gamma^\inv \T_{\K_T} \T_{\K_T}^\ast )^\inv \T_{\K_T} - \I_T }, \label{eq: interchange TK 1} \\
    &= \Tr\br{ \gamma^\inv \T_{\K_T}^\ast (  \I_T - \gamma^\inv \T_{\K_T} \T_{\K_T}^\ast )^\inv \T_{\K_T} },  \label{eq: interchange TK 2} \\
    &= \Tr\br{ (  \I_T - \gamma^\inv \T_{\K_T} \T_{\K_T}^\ast )^\inv  \gamma^\inv \T_{\K_T}   \T_{\K_T}^\ast}, \label{eq: interchange TK 3}
\end{align}
where \eqref{eq: interchange TK 1} is by \eqref{eq: matrix inversion lemma}, and 
\eqref{eq: interchange TK 3} is by the cyclic property of trace. Noting that the condition $\gamma \I_T \psdg \T_{\K_T}^\ast \T_{\K_T}$ is equivalent to $\gamma \I_T \psdg \T_{\K_T} \T_{\K_T}^\ast$, we expand $(  \I_T - \gamma^\inv \T_{\K_T} \T_{\K_T}^\ast )^\inv$ in \eqref{eq: interchange TK 3} by the following Neumann series:
\begin{equation}\label{eq:neuman series}
    (  \I_T - \gamma^\inv \T_{\K_T} \T_{\K_T}^\ast )^\inv = \sum_{k=0}^{\infty} \pr{\gamma^\inv \T_{\K_T} \T_{\K_T}^\ast}^k.
\end{equation}
Thus, the expression in \eqref{eq: interchange TK 3} can be written equivalently as
\begin{align}
    \Tr\br{ (  \I_T - \gamma^\inv \T_{\K_T} \T_{\K_T}^\ast )^\inv  \gamma^\inv \T_{\K_T}   \T_{\K_T}^\ast} &= \Tr\br{ \sum_{k=0}^{\infty} \pr{\gamma^\inv \T_{\K_T} \T_{\K_T}^\ast}^{k+1} }, \\
    &= \Tr\br{ \sum_{k=1}^{\infty} \pr{\gamma^\inv \T_{\K_T} \T_{\K_T}^\ast}^{k} },\\
    &= \Tr\br{ (  \I_T - \gamma^\inv \T_{\K_T} \T_{\K_T}^\ast )^\inv - \I_T },
\end{align}
giving the desired expression in \eqref{eq: stong duality finite horizon}. The desired expression in \eqref{eq:worst gamma strong duality finite horizon} can be obtained easily following similar algebraic manipulations.
\end{proof}

\paragraph{\textit{Proof of \cref{lem:finite horizon non-causal}:}}
Let $\K_T$ be a given estimator, which can be non-causal. We have that
\begin{align}
    \T_{\K_T} \T_{\K_T}^\ast &= \begin{bmatrix}  \K_T \HH_T - \L_T  & \K_T \end{bmatrix} \begin{bmatrix}  \HH_T^\ast \K_T^\ast  - \L_T^\ast  \\ \K_T^\ast \end{bmatrix},\\
    &= (\K_T \HH_T - \L_T) (\K_T \HH_T - \L_T)^\ast + \K_T \K_T^\ast, \\
    &= \K_T(\I_T + \HH_T \HH_T^\ast) \K_T^\ast - \K_T \HH_T \L_T^\ast  -   \L_T\HH_T^\ast \K_T^\ast + \L_T \L_T^\ast, \\
    &\stackrel{(a)}{=} (\K_T- \K_T^\circ)(\I_T + \HH_T \HH_T^\ast)(\K_T- \K_T^\circ)^\ast + \L_T \L_T^\ast -  \K_T^\circ (\I_T + \HH_T \HH_T^\ast) (\K_T^\circ)^\ast,
\end{align}
where $\K_T^\circ \defeq \L_T\HH_T^\ast (\I_T + \HH_T \HH_T^\ast)^\inv$ and $(a)$ is obtained from completion of squares.

Moreover, observe that
\begin{align}
     \T_{\K_T^\circ} \T_{\K_T^\circ}^\ast &=  \L_T \L_T^\ast -  \K_T^\circ (\I_T + \HH_T \HH_T^\ast) (\K_T^\circ)^\ast, \\
    &= \L_T \L_T^\ast -  \L_T\HH_T^\ast (\I_T + \HH_T \HH_T^\ast)^\inv \HH_T \L_T^\ast,\\
    &= \L_T (\I_T - \HH_T^\ast (\I_T + \HH_T \HH_T^\ast)^\inv \HH_T )\L_T^\ast, \\
    &\stackrel{(b)}{=} \L_T (\I_T + \HH_T^\ast \HH_T)^\inv\L_T^\ast,
\end{align}
where $(b)$ follows from matrix inversion identity.

Thus, we have that
\begin{equation}\label{eq:finite horizon TK in terms of non-causal}
     \T_{\K_T} \T_{\K_T}^\ast = (\K_T- \K_T^\circ)(\I_T + \HH_T \HH_T^\ast)(\K_T- \K_T^\circ)^\ast + \T_{\K_T^\circ} \T_{\K_T^\circ}^\ast \psdgeq \T_{\K_T^\circ} \T_{\K_T^\circ}^\ast
\end{equation}

Now, consider \cref{prob:finite horizon drf} without the causality constraint on the estimator. Using the strong duality result in \cref{thm: strong duality finite horizon}, we can express \cref{eq:finite horizon drf} equivalently as
\begin{equation}
    \inf_{ \gamma \geq 0 } \inf_{\K_T}  \gamma \rho_T^2 +  \gamma\Tr\br{ ( \I_T - \gamma^{\-1}\T_{\K_T} \T_{\K_T}^\ast )^\inv - \I_T } \quad  \textrm{s.t.} \quad  \gamma \I_T \psdg \T_{\K_T} \T_{\K_T}^\ast.
\end{equation}
Fixing $\gamma\geq 0$, we focus on the subproblem
\begin{equation}\label{eq: subproblem for noncaousal finite horizon}
    \inf_{\K_T} \gamma \Tr\br{ ( \I_T - \gamma^{\-1}\T_{\K_T} \T_{\K_T}^\ast )^\inv } \quad  \textrm{s.t.} \quad  \gamma \I_T \psdg \T_{\K_T} \T_{\K_T}^\ast.
\end{equation}
Using the identity in \eqref{eq:finite horizon TK in terms of non-causal}, we can rewrite \eqref{eq: subproblem for noncaousal finite horizon} in terms $\K_T^\circ$ as follows
\begin{equation*}
    \inf_{\K_T} \gamma^2 \Tr\br{ ( \gamma \I_T - (\K_T- \K_T^\circ)(\I_T + \HH_T \HH_T^\ast)(\K_T- \K_T^\circ)^\ast - \T_{\K_T^\circ} \T_{\K_T^\circ}^\ast)^\inv } \; \textrm{s.t.} \; \gamma \I_T \psdg \T_{\K_T} \T_{\K_T}^\ast.
\end{equation*}
Since the mapping $\clf{X} \mapsto (\I_T - \clf{X})^\inv$ is operator monotone, the minimum over $\K_T$ is attained by $\K_T^\circ$ and the optimal value is given by the following optimization:
\begin{equation}
     \inf_{ \gamma \geq 0 }  \gamma \rho_T^2 +  \gamma\Tr\br{ ( \I_T - \gamma^{\-1}\T_{\K_T^\circ} \T_{\K_T^\circ}^\ast )^\inv - \I_T } \quad  \textrm{s.t.} \quad  \gamma \I_T \psdg \T_{\K_T^\circ} \T_{\K_T^\circ}^\ast.
\end{equation}
\qed

\subsection{Proof of \texorpdfstring{\cref{thm:finite horizon sdp}}{Theorem 3.4}}

\begin{proof}
Using the strong duality result in \cref{thm: strong duality finite horizon}, we can express \cref{eq:finite horizon drf} equivalently as
\begin{equation}
    \inf_{\substack{ \gamma \geq 0 \\ \K_T\in \causal}}  \gamma( \rho_T^2 - \Tr(\I_T))  +  \gamma^2 \Tr\br{ ( \gamma \I_T - \T_{\K_T} \T_{\K_T}^\ast )^\inv } \quad  \textrm{s.t.} \quad  \gamma \I_T \psdg \T_{\K_T} \T_{\K_T}^\ast.
\end{equation}
Notice that we can express the rhs as
\begin{equation}
    \gamma^2 \Tr\br{ ( \gamma \I_T - \T_{\K_T} \T_{\K_T}^\ast )^\inv } = \inf_{\clf{X}_T \psdg 0} \Tr(\clf{X}_T) \;\subjto\; \clf{X}_T \psdgeq \gamma^2( \gamma \I_T - \T_{\K_T} \T_{\K_T}^\ast )^\inv.
\end{equation}
Using the Schur complement, we can rewrite the constraint $\clf{X}_T \psdgeq \gamma^2( \gamma \I_T - \T_{\K_T} \T_{\K_T}^\ast )^\inv$ as 
\begin{equation}\label{eq:matrix inequality for sdp}
    \begin{bmatrix}
        \clf{X}_T & \gamma \I_T \\
        \gamma \I_T & \gamma \I_T - \T_{\K_T} \T_{\K_T}^\ast
    \end{bmatrix}\psdgeq 0,
\end{equation}
where we used the fact that $\gamma \I_T \psdg \T_{\K_T} \T_{\K_T}^\ast$. Using the identity in \eqref{eq:finite horizon TK in terms of non-causal}, we can rewrite the matrix inequality \eqref{eq:matrix inequality for sdp} as
\begin{equation}
    \begin{bmatrix}
        \clf{X}_T & \gamma \I_T \\
        \gamma \I_T & \gamma \I_T - \T_{\K_T^\circ} \T_{\K_T^\circ}^\ast
    \end{bmatrix} -\begin{bmatrix}
        0 \\ (\K_T - \K_T^\circ) 
    \end{bmatrix} (\I_T + \HH_T \HH_T^\ast ) \begin{bmatrix}
        0 & (\K_T - \K_T^\circ)^\ast
    \end{bmatrix} \psdgeq 0.
\end{equation}
As $(\I_T + \HH_T \HH_T^\ast )\psdg 0$, by Schur complement theorem, we can reformulate the matrix inequality above as
\begin{equation}
    \begin{bmatrix}
         \mathcal{X}_T \!&\! \gamma \I_T \!&\! 0 \\
         \gamma \I_T  \!&\! \gamma \I_T \- \T_{\K_T^\circ}\T_{\K_T^\circ}^\ast  \!&\!  \K_T \- \K_T^\circ \\
         0 \!&\! (\K_T\-\K_T^\circ)^\ast \!&\! (\I_T\+\HH_T\HH_T^\ast)^\inv
        \end{bmatrix} \psdgeq 0.
\end{equation}
\end{proof}

\section{Proofs of Theorems Related to Infinite-Horizon Filtering}\label{app: infinite horizon proofs}

\subsection{Proof of \texorpdfstring{\cref{lem:infinite horizon non-causal}}{Lemma 3.7}}

\begin{theorem}[Strong Duality in the Infinite-Horizon]\label{thm: strong duality infinite horizon}
    Let $\K$ be a linear and time-invariant estimator (which can be non-causal in general) with bounded $\Hinf$ norm. Under the Assumptions~\ref{asmp:infinite horizon assumptions} and~\ref{asmp: nominal disturbance}, the infinite-horizon worst-case MSE \eqref{eq:infinite horizon worst case MSE} suffered by $\K$, \ie,
    \begin{equation}\label{eq: alternative infinite horizon worst case MSE}
        \overline{\wMSE}(\K,\rho) \,=\limsup_{T\to \infty} \frac{1}{T} \sup_{\Pr_T \in \W_T(\Pr_{T}^\circ,\rho_T)}  \E_{\Pr_T} \br{  \norm{\ee_T(\bm{\xi}_T,\K)}^2},
    \end{equation}
    attains a finite value and is equivalent to the following dual problem:
    \begin{equation}\label{eq: stong duality infinite horizon}
      \overline{\wMSE}(\K,\rho) \; = \inf_{\substack{\gamma \geq 0 } }  \gamma \rho^2 +  \gamma\Tr\br{ ( \I - \gamma^\inv \T_{\K} \T_{\K}^\ast )^\inv - \I } \quad  \textrm{s.t.} \quad  \gamma \I \psdg \T_{\K} \T_{\K}^\ast.
\end{equation}
Furthermore, the worst-case disturbance, $ \bm{\xi}_\star$, can be identified from the nominal disturbance, $ \bm{\xi}_\circ$, as 
\begin{equation}\label{eq:worst case distrubance infinite horizon}
    \bm{\xi}_\star = ( \I - \gamma_\star^\inv \T_{\K}^\ast \T_{\K})^{\-1}  \bm{\xi}_{\circ}, 
\end{equation}
where $\gamma_\star$ is the optimal solution of \eqref{eq: stong duality infinite horizon} and solves the following equation uniquely:
\begin{equation}\label{eq:worst gamma strong duality infinite horizon}
        \Tr\br{( (\I - \gamma_\star^\inv \T_{\K} \T_{\K}^\ast )^\inv - \I)^2} = \rho^2.
\end{equation}  
\end{theorem}
\begin{proof}
    \tk{The proof of this result closely tracks the proof of Thm. 5 in \cite{kargin_wasserstein_2023}. By replacing $\CC_\K$ in the proof of Thm. 5 in \cite{kargin_wasserstein_2023} with $\T_\K \T_{\K}^\ast$.}
\end{proof}

\paragraph{\textit{Proof of \cref{lem:infinite horizon non-causal}}}:
The proof follows closely from the proof of \cref{lem:finite horizon non-causal} in \cref{app:proof of non-causal finite horizon}. Let $\K$ be linear time-invariant estimator, which can be non-causal, with bounded $\Hinf$ norm. We have that
\begin{equation}\label{eq:infinite horizon TK in terms of non-causal}
     \T_{\K} \T_{\K}^\ast = (\K- \K_\circ)(\I + \HH \HH^\ast)(\K- \K_\circ)^\ast + \T_{\K_\circ} \T_{\K_\circ}^\ast \psdgeq \T_{\K_\circ} \T_{\K_\circ}^\ast
\end{equation}
where $\K^\circ \defeq \L\HH^\ast (\I + \HH \HH^\ast)^\inv$ and $ \T_{\K_T^\circ} \T_{\K_T^\circ}^\ast = \L (\I+ \HH^\ast \HH)^\inv\L^\ast$.

Now, consider \cref{prob:infinite horizon drf} without the causality constraint on the estimator. Using the strong duality result in \cref{thm: strong duality infinite horizon}, we can express \cref{eq:infinite horizon drf} equivalently as
\begin{equation}
    \inf_{ \gamma \geq 0 } \inf_{\K}  \gamma \rho^2 +  \gamma\Tr\br{ ( \I - \gamma^\inv \T_{\K} \T_{\K}^\ast )^\inv - \I } \quad  \textrm{s.t.} \quad  \gamma \I \psdg \T_{\K} \T_{\K}^\ast.
\end{equation}
Fixing $\gamma\geq 0$, we focus on the subproblem
\begin{equation}\label{eq: subproblem for noncaousal infinite horizon}
    \inf_{\K} \gamma \Tr\br{ ( \I - \gamma^\inv\T_{\K} \T_{\K}^\ast )^\inv } \quad  \textrm{s.t.} \quad  \gamma \I \psdg \T_{\K} \T_{\K}^\ast.
\end{equation}
Using the identity in \eqref{eq:infinite horizon TK in terms of non-causal}, we can rewrite \eqref{eq: subproblem for noncaousal infinite horizon} in terms $\K_\circ$ as follows
\begin{equation*}
    \inf_{\K} \gamma^2 \Tr\br{ ( \gamma \I - (\K- \K_\circ)(\I + \HH \HH^\ast)(\K- \K_\circ)^\ast - \T_{\K_\circ} \T_{\K_\circ}^\ast)^\inv } \; \textrm{s.t.} \; \gamma \I \psdg \T_{\K} \T_{\K}^\ast.
\end{equation*}
Since the mapping $\clf{X} \mapsto (\I - \clf{X})^\inv$ is operator monotone, the minimum over $\K$ is attained by $\K_\circ$ and the optimal value is given by the following optimization:
\begin{equation}
     \inf_{ \gamma \geq 0 }  \gamma \rho^2 +  \gamma\Tr\br{ ( \I - \gamma^\inv\T_{\K_\circ} \T_{\K_\circ}^\ast )^\inv - \I } \quad  \textrm{s.t.} \quad  \gamma \I \psdg \T_{\K_\circ} \T_{\K_\circ}^\ast.
\end{equation}
\qed


\subsection{Proofs of \texorpdfstring{\cref{thm:dual formulation}}{Theorem 3.8} and  \texorpdfstring{\cref{thm:stability}}{Corollary 3.9}} \label{app:infinite horizon full optimality}

\begin{lemma}[{Wiener-Hopf Method \cite{kailath_linear_2000}}]\label{lem:wiener-hopf}
For a bounded and positive definite Toeplitz operator $\M\psdg 0$, let $\M \mapsto \Phi(\M)$ be a mapping defined as 
\begin{equation}\label{eq:wiener-hopf function_method}
    \Phi(\M) \triangleq \inf_{\K \in \causal} \Tr\pr{ \T_{\K} \T_{\K}^\ast \M }.
\end{equation}
Denote by $\M = \U^\ast \U$ and $\Delta \Delta^\ast = \I + \HH \HH^\ast$ the canonical spectral factorizations\footnote{The canonical spectral factorization is essentially the Toeplitz operator counterpart of Cholesky decomposition of finite-dimensional matrices.} where $\U$, $\Delta$ as well as their inverses $\U^\inv$, $\Delta^\inv$ are causal operators. The following statements hold:
\begin{itemize}
    \item[i.] The optimal causal solution to \eqref{eq:wiener-hopf function} is given by 
    \begin{equation}
            \K = \U^{-1} \cl{ \U \K_\circ \Delta }_{+} \Delta^\inv = \K_{\Htwo} +  \U^{-1}\cl{ \U \{\K_\circ \Delta\}_{-} }_{+} \Delta^\inv,
    \end{equation}
    where $\K_{\Htwo} \defeq  \cl{\K_\circ \Delta }_{+} \Delta^\inv $ is the Kalman filter. 
    \item[ii.] The function $\Phi$ can be written in closed form as 
    \begin{equation}
        \Phi(\M) = \Tr\br{ \cl{ \U \K_\circ \Delta }_{-}\cl{ \U \K_\circ \Delta }_{-}^\ast} + \Tr\pr{\T_{\K_\circ}\T_{\K_\circ}^\ast \M },
    \end{equation}
    where $\T_{\K_\circ}\T_{\K_\circ}^\ast = \L(\I+\HH^\ast \HH)^\inv \L^\ast$.
    \item[iii.] The gradient of $\Phi$ has the following closed form
    \begin{equation}\label{eq: grad Phi}
        \nabla \Phi(\M) = \T_{\K} \T_{\K}^\ast = \U^{-1} \cl{ \U \K_\circ \Delta }_{-}\cl{\U \K_\circ \Delta }_{-}^\ast \U^{-\ast} + \T_{\K_\circ}\T_{\K_\circ}^\ast .
    \end{equation}
\end{itemize}
\end{lemma}
\begin{proof}
        Using the identity \eqref{eq:infinite horizon TK in terms of non-causal} and the cyclic property of $\Tr$, the objective can be written as,
    \begin{align}\label{eq:wiener-hopf function_eqn}
        \inf_{\K \in\causal} \Tr\br{ (\K - \K_\circ ) \Delta \Delta^\ast (\K - \K_\circ )^{\ast}\M }  &= \inf_{\K \in\causal} \Tr\br{ (\K\Delta - \K_\circ\Delta ) (\K \Delta - \K_\circ \Delta )^{\ast}\U^\ast \U } \\
        &= \inf_{\K \in\causal} \inf_{\K \in\causal} \Tr\br{ (\U\K\Delta - \U\K_\circ\Delta ) (\U\K \Delta - \U\K_\circ \Delta )^{\ast}  } \\
        &= \inf_{\K \in\causal} \Norm[\Htwo]{\U\K \Delta - \U\K_\circ \Delta }^2 ,
    \end{align}
 where $\|\cdot\|_2$ represents the $\Htwo$ norm. 
 Since $\Delta, \K$ and $\U$ are causal, and $\U \K_\circ \Delta$ can be broken into causal and non-causal parts, it is evident that the (causal) filter that minimises the objective is the one that makes the term $\U\K \Delta - \U\K_\circ \Delta$ strictly anti-causal, cancelling off the causal part of $\Delta\K_\circ \L$. This means that the optimal filter satisfies,
 \begin{align}
    \U\K \Delta = \cl{\U\K_\circ\Delta}_{\!+}.
\end{align}
Also, since $\U^\inv$ and $\Delta^\inv$ are causal, the optimal causal filter is given by 
\begin{equation}
    \K = \U^{-1} \cl{ \U \K_\circ \Delta }_{+} \Delta^\inv.
\end{equation}
Furthermore, using the identity $ \K_\circ \Delta = \{ \K_\circ \Delta\}_{+}+ \{ \K_\circ \Delta\}_{-}$, we get 
\begin{align}
    \K &= \U^{-1} \cl{ \U \{\K_\circ \Delta\}_{+} }_{+} \Delta^\inv + \U^{-1} \cl{ \U \{\K_\circ \Delta\}_{-} }_{+} \Delta^\inv,\\
    &= \U^{-1} \U \{\K_\circ \Delta\}_{+} \Delta^\inv + \U^{-1} \cl{ \U \{\K_\circ \Delta\}_{-} }_{+} \Delta^\inv,\\
    &= \{\K_\circ \Delta\}_{+} \Delta^\inv + \U^{-1} \cl{ \U \{\K_\circ \Delta\}_{-} }_{+} \Delta^\inv.
\end{align}
Plugging this solution to $ \T_\K \T_\K^\ast$, we get 
\begin{align}
    \T_\K \T_\K^\ast &= \U^\inv  (\U\K \Delta - \U\K_\circ \Delta )   (\U\K \Delta - \U\K_\circ \Delta )^\ast \U^{\-\ast} + \T_{\K_\circ} \T_{\K_\circ}^\ast,\\
    &=\U^\inv(\cl{\U\K_\circ\Delta}_{\!+}- \U\K_\circ \Delta )   (\cl{\U\K_\circ\Delta}_{\!+} - \U\K_\circ \Delta )^\ast \U^{\-\ast} + \T_{\K_\circ} \T_{\K_\circ}^\ast, \\
    &=\U^\inv\cl{\U\K_\circ\Delta}_{\!-}   \cl{\U\K_\circ\Delta}_{\!-}^\ast \U^{\-\ast}  + \T_{\K_\circ} \T_{\K_\circ}^\ast. \label{eq: grad of Phi from tk}
\end{align}
Then, the objective becomes
\begin{align}
    \Tr(\T_\K \T_\K^\ast \M) &=\Tr(\U^\inv\cl{\U\K_\circ\Delta}_{\!-}   \cl{\U\K_\circ\Delta}_{\!-}^\ast \U^{\-\ast}\M)  + \Tr(\T_{\K_\circ} \T_{\K_\circ}^\ast \M),\\
    &=\Tr(\cl{\U\K_\circ\Delta}_{\!-}   \cl{\U\K_\circ\Delta}_{\!-}^\ast )  + \Tr(\T_{\K_\circ} \T_{\K_\circ}^\ast \M).
\end{align}
Finally, by Danskin theorem \cite{danskin}, the gradient of $\Phi$ is simply $\T_\K \T_\K^\ast$ evaluated at the optimal $\K$ as given in \eqref{eq: grad of Phi from tk}.
\end{proof}

\begin{lemma}\label{thm: fenchel dual}
Let $\gamma > \inf_{\K\in\causal} \norm[\op]{\T_{\K}}^2$ be fixed. Then, we have the following duality
\begin{equation}\label{eq:suboptimal prob}
    \inf_{\substack{\K \in \causal,\\ \gamma \I \psdg \T_\K \T_\K^\ast}} \gamma\Tr\br{ ( \I \- \gamma^\inv \T_{\K} \T_{\K}^\ast )^\inv \- \I }  \=  \sup_{\M \psdg 0}  \inf_{\K \in\causal} \Tr(\T_{\K} \T_{\K}^\ast\M ) \- \gamma\Tr\pr{\M \- 2\sqrt{\M}\+\I} .
\end{equation}
\end{lemma}
\begin{proof}
The convex mapping $\clf{X}\! \mapsto\! \Tr \clf{X}^{\inv}$ for $\clf{X}\!\psdg\! 0$ can be expressed via Fenchel duality as
\begin{equation}\label{eq:trace_of_inverse}
   \sup_{\M \psdg 0 } -\Tr(\clf{X} \M ) + 2\Tr(\sqrt{\M}) = \begin{cases} \Tr(\clf{X}^{-1}), & \quad\text{if } \clf{X}\psdg 0 \\  +\infty, &\quad \text{o.w.} \end{cases}
\end{equation}
Using the identity~\eqref{eq:trace_of_inverse}, we rewrite the original problem as, 
\begin{equation}\label{eq:suboptimal prob_rewrite}
       \inf_{\K \in\causal} \sup_{\M \psdg 0}  \Tr(\T_{\K} \T_{\K}^\ast\M ) - \gamma\Tr\pr{\M - 2\sqrt{\M}+\I} .
\end{equation}
Notice that, the objective above is strictly convex in $\K$ and strictly concave in $\M$. Furthermore, the primal and dual problems are feasible since $\gamma > \inf_{\K\in\causal} \norm[\op]{\T_{\K}}^2$. Thus, the proof follows from the minimax theorem. 
\end{proof}

\paragraph{\textit{Proof of \cref{thm:dual formulation}}:}
Consider \cref{prob:infinite horizon drf}. Using the strong duality result in \cref{thm: strong duality infinite horizon}, we can express \cref{eq:infinite horizon drf} equivalently as
\begin{equation}\label{eq: full optimality}
    \inf_{ \gamma \geq 0 } \inf_{\K\in\causal}  \gamma \rho^2 +  \gamma\Tr\br{ ( \I - \gamma^\inv \T_{\K} \T_{\K}^\ast )^\inv - \I } \quad  \textrm{s.t.} \quad  \gamma \I \psdg \T_{\K} \T_{\K}^\ast.
\end{equation}
Fixing $\gamma\geq 0$, we focus on the subproblem
\begin{equation}\label{eq: subproblem for causal infinite horizon}
    \inf_{\K} \gamma \Tr\br{ ( \I - \gamma^\inv\T_{\K} \T_{\K}^\ast )^\inv -\I } \quad  \textrm{s.t.} \quad  \gamma \I \psdg \T_{\K} \T_{\K}^\ast.
\end{equation}
Using \cref{thm: fenchel dual}, we can reformulate \eqref{eq: subproblem for causal infinite horizon} as 
\begin{equation}
    \sup_{\M \psdg 0}  \inf_{\K \in\causal} \Tr(\T_{\K} \T_{\K}^\ast\M ) \- \gamma\Tr\pr{\M \- 2\sqrt{\M}\+\I}.
\end{equation}
Thus, the original formulation in \eqref{eq: full optimality} can be expressed as
\begin{equation}
    \inf_{ \gamma \geq 0 } \sup_{\M \psdg 0}  \inf_{\K \in\causal}    \Tr(\T_{\K} \T_{\K}^\ast\M ) + \gamma\pr{\rho^2- \Tr(\M - 2\sqrt{\M}+\I)}.
\end{equation}
Note that the objective above is affine in $\gamma\geq0$ and strictly concave in $\M$. Moreover, primal and dual feasibility hold, enabling the exchange of $ \inf_{ \gamma \geq 0 } \sup_{\M \psdg 0} $ resulting in
\begin{align}
    \sup_{\M \psdg 0}  \inf_{\K \in\causal}   \inf_{ \gamma \geq 0 }  \Tr(\T_{\K} \T_{\K}^\ast\M ) + \gamma\pr{\rho^2- \Tr(\M - 2\sqrt{\M}+\I)},
\end{align}
where the inner minimization over $ \gamma$ reduces the problem to its constrained version in \cref{eq:infinite horizon drf convex}.

Finally, the form of the optimal $\K_\star$ follows from the Wiener-Hopf technique in \cref{lem:wiener-hopf} and the optimal $\gamma_\star$ and $\M_\star$ can be obtained using the strong duality theorem in \eqref{thm: strong duality infinite horizon}.

\paragraph{\textit{{Proof of \cref{thm:stability}}}:} This result follows immediately from the finiteness of the time-averaged infinite-horizon MSE. 



\section{Additional Discussion on Frequency-domain Optimization Method}\label{ap:numerical method}
\subsection{Pseudocode for Frequency-domain Iterative Optimization Method Solving \texorpdfstring{\cref{eq:infinite horizon drf convex}}{Eq. (14)}}\label{app: detailed pseudocode }

\begin{algorithm}[ht]
   \caption{Frequency-domain iterative optimization method solving \cref{eq:infinite horizon drf convex}}
   \label{alg:fixed_point_detailed}
\begin{algorithmic}[1]
   \STATE {\bfseries Input:} Radius $\rho\>0$, state-space model $(A,B,C_y,C_s)$, discretization $N\>0$, tolerance $\epsilon\>0$ 
   \vspace{0.5mm}
   \STATE Compute $(\overline{A},\overline{B},\overline{C})$ from $(A,B,C_y,C_s)$ using \eqref{eq: modified A B C}
   \vspace{0.5mm}
   \STATE Generate frequency samples $\TT_N \defeq \{\e^{j 2\pi n /N} \mid n \!=\! 0,\dots,N\!-\!1\}$
   \vspace{0.5mm}
   \STATE Initialize $M_{0}(z) \gets I$ for $z\in\TT_N$, and $k\gets 0$
   \vspace{0.5mm}
   \REPEAT 
   \vspace{0.5mm}
   \STATE Set the step size $\eta_{k}\gets \frac{2}{k+2}$
   \vspace{0.5mm}
   \STATE Compute the spectral factor $\begin{aligned}U_{k}(z)\gets \texttt{SpectralFactor}(M_{k})\end{aligned}$ (see \cref{alg:spectral factor method})
   \vspace{0.5mm}
   \STATE Compute the parameter $\begin{aligned}\Gamma_{k} \gets \frac{1}{N}\suml_{z\in\TT_N} U_{k}(z) \overline{C}(I-z\overline{A})^{-1} \end{aligned}$ (see \cref{app: gradients in fw})
   \vspace{1mm}
   \STATE Compute the gradient for $z\in\TT_N$ (see \cref{app: gradients in fw})\\
   \vspace{0.5mm}
    $\begin{aligned}\quad\quad G_{k}(z) \gets U_{k}(z)^{\!-\!1} \Gamma_{k} (I\!-\!z\overline{A})^{\!-\!1}\overline{B}\, \overline{B}^\ast (I\!-\!z\overline{A})^{\!-\!\ast} \Gamma_{k}^\ast U_{k}(z)^{\!-\!\ast} \!+\! T_{K_\circ}(z)T_{K_\circ}(z)^\ast\end{aligned}$
    \vspace{1mm}
   \STATE Solve the linear subproblem \eqref{eq:fw linear subproblem} via bisection (see \cref{app:bisection})\\
   \vspace{0.5mm}
    $\begin{aligned}\quad\quad \widetilde{M}_{k}(z) \gets \texttt{Bisection}(G_k, \rho, \epsilon)\end{aligned}$ for $z\in\TT_N$.
    \vspace{0.5mm}
    \STATE Set $M_{k+1}(z)\gets (1-\eta_k) M_k(z) + \eta_k \widetilde{M}_{k}(z) $ for $z\in\TT_N$.
    \vspace{0.5mm}
   \STATE Increment $k\gets k+1$
   \vspace{0.5mm}
   \UNTIL{$\norm{M_{k+1} - M_{k}}/\norm{M_{k}} \leq \epsilon $}
   \vspace{0.5mm}
   \STATE Compute $K(z) \gets \texttt{RationalApproximate}(M_{k+1})$ (see \cref{app: rational approximation})
\end{algorithmic}
\end{algorithm}

\subsection{Frequency-Domain Characterization of the Optimal Solution of  \texorpdfstring{\cref{eq:infinite horizon drf}}{Eq. (11)}} \label{app: frequency domain details}

We present the frequency-domain formulation of the saddle point  $(\K_\star,\M_\star)$ derived in \cref{thm:dual formulation} to reveal the structure of the solution. We first introduce the following useful results:

\begin{lemma}[{\cite[pg. 261]{blackbook}}]\label{lem:delta  factorization}
    Given $H(z) \defeq C_y (zI - A)^\inv B $, consider the canonical spectral factorization $\Delta(z)\Delta(z)^\ast I + H(z)H(z)^\ast$ for $z\in\TT$. We have that
    \begin{align}
        \Delta(z) &= ( I + C_y (zI-A)^\inv F_P) R_e^{1/2}, \\
        \Delta(z)^\inv &= R_e^{-1/2} ( I - C_y (zI-A_P)^\inv F_P), 
    \end{align}
    where $R_e \defeq I + C_y P C_y^\ast$, $F_P \defeq (A P C_y^\ast) R_e^\inv$, $A_P \defeq A - F_P C_y$, and $P$ is the unique positive semidefinite solution to the following discrete algebraic Riccati equation (DARE)
    \begin{equation}\label{eq:dare}
        P = A P A^\ast + B B^\ast - F_P R_e F_P^\ast.
    \end{equation}
\end{lemma}

Denoting by $M_\star(z)$ and $T_{K_\star}(z)$ the transfer functions corresponding to the optimal $\M_\star$ and $\T_{\K_\star}$, respectively, the optimality condition in \eqref{eq:optimal K and M} takes the equivalent form:
\begin{subequations}\label{eq: frequency domain optimal M and K}
    \begin{align}
    &\textit{i.}\; M_\star(z) =  \pr{ I - \gamma_\star^\inv T_{K_\star} (z) T_{K_\star} (z)^\ast }^{-2},  \label{eq: optimal M(z)}\\
    &\textit{ii.}\; T_{K_\star}(z) T_{K_\star}(z)^\ast \= U_\star(z)^{\-1} \cl{ \U_\star \Ss }_{\-}(z) \cl{\U_\star \Ss}_{\-}(z)^\ast U_\star(z)^{\-\ast} \+ T_{K_\circ}(z)T_{K_\circ}(z)^\ast,  \label{eq: optimal K(z)}\\
    &\textit{iii.}\; \Tr\br{\pr{(I - \gamma_\star^\inv T_{K_\star} (z) T_{K_\star} (z)^\ast)^\inv - I}^2 } = \rho^2,  \label{eq: optimal gamma star}
\end{align}
\end{subequations}
where $\Ss \defeq \{\K_\circ \Delta \}_{-}$ is a strictly anticausal operator and $U_\star(z)$ is the transfer function corresponding to the causal canonical factor $\U_\star$

The transfer function corresponding to the operator $\Ss$ takes a rational form as 
\begin{equation}\label{eq: S(z)}
    S(z) \defeq \overline{C}(z^\inv I - \overline{A})^\inv \overline{B},
\end{equation}
where $(\overline{A},\overline{B},\overline{C})$ are determined by the original state-space parameters $(A,B,C_y,C_s)$. The following lemma explicitly states this result. 
\begin{lemma}[{\cite[pg. 261]{blackbook} and \cite[Lem. 6]{sabag_regret-optimal_2022}}]\label{lem: state space decomposition of non-causal}
We have that 
\begin{equation}
    \K_\circ \Delta = \K_{\Htwo} \Delta  + \Ss ,
\end{equation}
where $\K_{\Htwo}$ is the nominal causal $\Htwo$ (aka Kalman) filter and $\Ss \defeq \{\K_\circ \Delta \}_{-}$ is strictly anti-causal. Furthermore, the corresponding transfer functions take an explicit form as highlighted below
\begin{align}\label{eq: state space decomposition}
    K_{\Htwo}(z) &\defeq C_s P C_y^\star R_e^\inv + C_s (I-PC_y^\ast R_e^\inv C_y)(zI - A_P)^\inv F_P,\\
    S(z) &\defeq C_s P A_P^\ast (z^\inv I - A_P^\ast)^\inv C_y^\ast R_e^{-1/2},
\end{align}
where $P$, $R_e$, $F_P$, and $A_P$ are defined as in \cref{lem:delta  factorization}.
\end{lemma}

Thus, we have that
\begin{equation}\label{eq: modified A B C}
    \overline{A} \triangleq A_P^\ast, \quad \overline{B} \triangleq C_y^\star R_e^{-1/2}, \quad \overline{C}\triangleq C_s P A_P^\ast .
\end{equation}

Notice that for a causal $U(z)$ and strictly anti-causal $S(z)$, the strictly anti-causal part $\{U(z)S(z)\}_{-}$ may not have any poles from $U(z)$, and all of its poles must be from the strictly anti-causal $S(z)$. This observation is formally expressed in the following lemma.
\begin{lemma}\label{eq: finite parameter of anticausal (U S) }
Let $\U$ be a causal and causally invertible operator, which can be non-rational in general. Then, the strictly anti-causal operator $\cl{\U \Ss}_{\-}$ admits a rational transfer function, \ie,
    \begin{equation}
        \cl{\U \Ss}_{\-}(z) = \Gamma (z^\inv I - \overline{A})^\inv \overline{B},
    \end{equation}
    where 
    \begin{equation}
         \Gamma \triangleq \frac{1}{2\pi} \int_{-\pi}^{\pi}  U(\ejw) \overline{C} (I-\ejw \overline{A})^\inv d\omega.
    \end{equation}
\end{lemma}
\begin{proof}
    Consider the z-transform expansions of $U(z)$ and $S(z)$:
    \begin{equation}
        U(z) = \sum_{k=0}^{\infty} \widehat{U}_k z^{-k}, \quad \textrm{and} \quad S(z) = \sum_{l=0}^{\infty} \overline{C} \,\overline{A}^l \,\overline{B} z^{l+1},
    \end{equation}
    where the time-domain coefficients $\widehat{U}_k$ can be derived from the Fourier series integrals as
    \begin{equation}
         \widehat{U}_k \defeq \frac{1}{2\pi} \int_{-\pi}^{\pi} U(\ejw) \e^{j \omega k} d\omega.
    \end{equation}
    Multiplying $U(z)$ and $S(z)$ and taking the strictly anti-causal parts, \ie, terms with positive powers of $z$, we get
    \begin{align}
        \{U(z) S(z)\}_{-} &= \cl{\pr{\sum_{k=0}^{\infty} \widehat{U}_k z^{-k}} \pr{\sum_{l=0}^{\infty} \overline{C} \,\overline{A}^l \,\overline{B} z^{l+1}} }_{-}, \\
        &= \pr{\sum_{k=0}^{\infty} \widehat{U}_k\overline{C} \, \overline{A}^k } \pr{\sum_{l=0}^{\infty} \overline{A}^l \,\overline{B} z^{l+1}},\\
        &= \Gamma (z^\inv I - \overline{A})^\inv \overline{B},
    \end{align}
    where $\Gamma =\sum_{k=0}^{\infty} \widehat{U}_k\overline{C} \, \overline{A}^k  $ which can be expressed as an integral
    \begin{equation}
        \Gamma = \frac{1}{2\pi} \int_{-\pi}^{\pi}  U(\ejw) \overline{C} (I-\ejw \overline{A})^\inv d\omega.
    \end{equation}
    using Parseval's theorem. 
\end{proof}


\subsection{Proofs of  \texorpdfstring{\cref{thm:fixed_point}}{Lemma 4.1} and \texorpdfstring{\cref{cor:non-rational}}{Corollary 4.2}}

\paragraph{\textit{Proof of \cref{thm:fixed_point}:}}
Using \cref{eq: finite parameter of anticausal (U S) }, the frequency-domain optimality equations \eqref{eq: frequency domain optimal M and K} can be reformulated explicitly as follows
\begin{subequations}\label{eq: frequency domain optimal M and K alternative }
    \begin{align}
    &\textit{i.}\; M_\star(z) =  \pr{ I - \gamma_\star^\inv T_{K_\star} (z) T_{K_\star} (z)^\ast }^{-2},  \label{eq: optimal M(z) alt}\\
    &\textit{ii.}\; T_{K_\star}(z) T_{K_\star}(z)^\ast \= U_\star(z)^{\-1} \Gamma_\star ( I \- z\overline{A})^\inv \overline{B} \,\overline{B}^\ast ( I \- z\overline{A})^{\-\ast} \Gamma_\star^\ast U_\star(z)^{\-\ast} \+ T_{K_\circ}(z)T_{K_\circ}(z)^\ast,  \label{eq: optimal K(z) alte}\\
    &\textit{iii.}\; \Tr\br{\pr{(I - \gamma_\star^\inv T_{K_\star} (z) T_{K_\star} (z)^\ast)^\inv - I}^2 } = \rho^2,  \label{eq: optimal gamma star alternative}
\end{align}
\end{subequations}
where
\begin{equation}
     \Gamma_\star \triangleq \frac{1}{2\pi} \int_{-\pi}^{\pi}  U_\star(\ejw) \overline{C} (I-\ejw \overline{A})^\inv d\omega,
\end{equation}
and $(\overline{A},\overline{B},\overline{C})$ are as in \eqref{eq: modified A B C}. This gives us the desired result.
\qed

\paragraph{\textit{Proof of \cref{cor:non-rational}:}}
Define $S_{\star}(z) \defeq \Gamma_\star ( I \- z\overline{A})^\inv \overline{B}$ for notational convenience. We rewrite the optimality conditions in \eqref{eq: frequency domain optimal M and K alternative } as
\begin{align}
    &\textit{i.}\;(U_\star(z)^\ast  U_\star(z))^{-1/2} =   I - \gamma_\star^\inv T_{K_\star}(z) T_{K_\star}(z)^\ast\\
    &\textit{ii.}\;  T_{K_\star}(z) T_{K_\star}(z)^\ast \= U_\star(z)^{\-1}S_\star(z ) S_\star(z )^\ast  U_\star(z)^{\-\ast} \+ T_{K_\circ}(z)T_{K_\circ}(z)^\ast
\end{align}
By plugging $\textit{ii.}$ into $\textit{i.}$, we get
\begin{align}
      0=I - (U_\star(z)^\ast  U_\star(z))^{-1/2} -  \gamma_\star^\inv \pr{U_\star(z)^{\-1}S_\star(z ) S_\star(z )^\ast  U_\star(z)^{\-\ast} \+ T_{K_\circ}(z)T_{K_\circ}(z)^\ast} = 0,
\end{align}
Multiplying by $U_\star(z)$ from the left and by $U_\star(z)^\ast$ from the right, we get 
\begin{align*}
   0&= U_\star(z) U_\star(z)^\ast \-(U_\star(z) U_\star(z)^\ast)^{1/2} \-  \gamma_\star^\inv \pr{S_\star(z ) S_\star(z )^\ast  \+ U_\star(z)T_{K_\circ}(z)T_{K_\circ}(z)^\ast} U_\star(z)^\ast,
\end{align*}
which can be written further as
\begin{align}
    U_\star(z) U_\star(z)^\ast  = \frac{1}{4}\pr{I + \sqrt{ I + 4\gamma^\inv \pr{S_\star(z ) S_\star(z )^\ast  \+ U_\star(z)T_{K_\circ}(z)T_{K_\circ}(z)^\ast} U_\star(z)^\ast  }}^2.
\end{align}
Notice that while $S_\star(z ) S_\star(z )^\ast $ is rational, the expression above involves its positive definite square root, which does not generally preserve rationality, implying the desired result.
\qed

\subsection{Additional Discussion on the Computation of Gradients}\label{app: gradients in fw}
By the Wiener-Hopf technique discussed in \cref{lem:wiener-hopf}, the gradient $\G_k = \nabla \Phi (\M_k)$ can be obtained as
\begin{align}
    G_k(z) = U_k(z)^{-1} \cl{ \U_k \K_\circ \Delta }_{-}(z)\cl{\U_k \K_\circ \Delta }_{-}(z)^\ast U_k^{-\ast} + T_{K_\circ}T_{K_\circ}^\ast ,
\end{align}
where $\U_k^\ast \U_k = \M_k$ is the unique spectral factorization. Furthermore, by \cref{eq: finite parameter of anticausal (U S) }, we can reformulate the gradient $ G_k(z)$ more explicitly as
\begin{align}
    G_k(z) = U_k(z)^{-1}  \Gamma_k ( I \- z\overline{A})^\inv \overline{B} \,\overline{B}^\ast ( I \- z\overline{A})^{\-\ast} \Gamma_k^\ast U_k^{-\ast} + T_{K_\circ}T_{K_\circ}^\ast ,
\end{align}
where
\begin{equation}
         \Gamma_k \triangleq \frac{1}{2\pi} \int_{-\pi}^{\pi}  U_k(\ejw) \overline{C} (I-\ejw \overline{A})^\inv d\omega.
\end{equation}
Here, the spectral factor $U_k(z)$ is obtained for $z\in\TT_N$ by \cref{alg:spectral factor method}. Similarly, the parameter $\Gamma_k$ can be computed numerically using the trapezoid rule over the discrete domain $\TT_N$, \ie,
\begin{equation}
         \Gamma_k \gets \frac{1}{N} \sum_{z\in \TT_N}  U_k(z) \overline{C} (I-z \overline{A})^\inv.
\end{equation}
Noting that $T_{K_\circ}T_{K_\circ}^\ast$ is rational and depends only on the system, the gradient $G_k(z)$ can be efficiently computed for $z\in\TT_N$.

\subsection{Implementation of Spectral Factorization}\label{app: spectral factorization}

To perform the spectral factorization of an irrational function $M(z)$, we use a spectral factorization method via discrete Fourier transform, which returns samples of the spectral factor on the unit circle. First, we compute $\Lambda(z)$ for $z\in\TT_N$, which is defined to be the logarithm of $M(z)$, then we take the inverse discrete Fourier transform $\lambda_k$ for $k=0,\dots,N-1$ of $\Lambda(z)$ which we use to compute the spectral factorization as $$U(z_n) \gets \exp\pr{\frac{1}{2} \lambda_0 + \sum_{k=1}^{N/2-1} \lambda_k z_n^{-k} + \frac{1}{2} (-1)^{n} \lambda_{N/2}}$$ for $k=0,\dots,N-1$ where $z_n= \e^{j 2\pi n /N}$ .

The method is efficient without requiring rational spectra, and the associated error term, featuring a purely imaginary logarithm, rapidly diminishes with an increased number of samples. It is worth noting that this method is explicitly designed for scalar functions.

\begin{algorithm}[ht]\label{alg:spectral factor method}
   \caption{\texttt{SpectralFactor}: Spectral Factorization via DFT}
\begin{algorithmic}[1]
   \STATE {\bfseries Input:} Scalar positive spectrum $M(z)>0$ on  $\TT_N \defeq \{\e^{j 2\pi n /N} \mid n \!=\! 0,\dots,N\!-\!1\}$ 
   \vspace{0.5mm}
   \STATE {\bfseries Output:} Causal spectral factor $U(z)$ of $M(z)>0$ on  $\TT_N$ 
   \vspace{0.5mm}
   \STATE Compute the cepstrum $\begin{aligned}\Lambda(z) \gets \log(M(z))\end{aligned}$ on $z\in \TT_N$.
   \vspace{0.5mm}
   \STATE Compute the inverse DFT  \\
    \vspace{0.5mm}
    $\begin{aligned}\lambda_k \gets \operatorname{IDFT}(\Lambda(z))\end{aligned}$ for $k=0,\dots,N\!-\!1$
   \vspace{0.5mm}
   \STATE Compute the spectral factor for $z_n= \e^{j 2\pi n /N}$ \\
    \vspace{0.5mm}
   $\begin{aligned}U(z_n) \gets \exp\pr{\frac{1}{2} \lambda_0 + \sum_{k=1}^{N/2-1} \lambda_k z_n^{-k} + \frac{1}{2} (-1)^{n} \lambda_{N/2}} \end{aligned}$, \quad $n=0,\dots,N\!-\!1$
\end{algorithmic}
\end{algorithm}

\subsection{Implementation of Bisection Method}\label{app:bisection}

To find the optimal parameter $\gamma_k$ that solves $\Tr\br{((I\-\gamma_k^\inv \G_k)^{-1} \- I)^2} \= \rho^2$ in the Frank-Wolfe update \eqref{eq:frank wolfe frequency}, we use a bisection algorithm. The pseudo code for the bisection algorithm can be found in Algorithm \ref{alg:bisection}. We start off with two guesses of $\gamma$ \ie ($\gamma_{left}, \gamma_{right}$) with the assumption that the optimal $\gamma$ lies between the two values (without loss of generality).
\begin{algorithm}[H]
\caption{\texttt{Bisection Algorithm}}\label{alg:bisection}
\begin{algorithmic}[1]
   \STATE {\bfseries Input: $\gamma_{right}, \gamma_{left}$} 
   
   \vspace{0.5mm}
   \STATE Compute the gradient at $\gamma_{right}$: $grad\_\gamma_{right}$

   \vspace{0.5mm}
   
   \vspace{0.5mm}
   \WHILE {$\mid \gamma_{right} - \gamma_{left} \mid > \epsilon$}
       \STATE Calculate the midpoint $\gamma_{mid}$ between $\gamma_{left}$ and $\gamma_{right}$
       \STATE Compute the gradient at $\gamma_{mid}$
       
       \vspace{0.5mm}
       \IF {the gradient at $\gamma_{mid}$ is zero}
           \STATE \textbf{return} $\gamma_{mid}$ \COMMENT{Root found}
       \ELSIF {the gradient at $\gamma_{mid}$ is positive}
           \STATE Update $\gamma_{right}$ to $\gamma_{mid}$
       \ELSE
           \STATE Update $\gamma_{left}$ to $\gamma_{mid}$
       \ENDIF
   \ENDWHILE
   
   \vspace{0.5mm}
   \STATE \textbf{return} the average of $\gamma_{left}$ and $\gamma_{right}$ \COMMENT{Approximate root}
\end{algorithmic}
\end{algorithm}

\subsection{Proof of \texorpdfstring{\cref{thm: convergence of FW}}{Theorem 4.4}}

Our proof of convergence follows closely from the proof technique used in \cite{jaggi_revisiting_2013}. In particular, since the unit circle is discretized and the computation of the gradients $G_k(z)$ are approximate, the linear suboptimal problem is solved up to an approximation, $\delta_N$, which depends on the problem parameters, and the discretization level $N$. Namely, 
\begin{align}
    \Tr(\nabla \Phi(\M_k)  \widetilde{\M}_{k+1}) \geq   \sup_{\substack{\M \in  \Omega_\rho}} \Tr(\nabla \Phi(\M_k)  {\M})  - \delta_N
\end{align}
where
\begin{align}
\Omega_\rho \defeq \{\M\psdg 0 \mid \Tr(\M-2\sqrt{\M}+\I) \leq \rho^2\} ,   
\end{align}
Therefore, using Theorem 1 of \cite{jaggi_revisiting_2013}, we obtain 
\begin{equation}\label{eq:convergence rate_app}
    \vspace{-2mm}
        \Phi(\M_\star) - \Phi(\M_k) \leq \frac{2\kappa}{k+2}(1+\delta_N).
        \vspace{-2mm}
\end{equation}
where
\begin{align}
    \kappa \defeq \sup_{\substack{\M, \widetilde{\M} \in \Omega_\rho \\ \eta \in [0,1] \\ \M^\prime = \M + \eta(\widetilde{\M}-\M) }} \frac{2}{\eta^2} \pr{\Tr(\M^\prime \nabla \Phi(\M))-\Phi(\M^\prime)}.
\end{align}

\subsection{Implementation of Rational Approximation}\label{app: rational approximation}
We present the pseudocode of \texttt{RationalApproximation}.
\begin{algorithm}[ht]\label{alg:rational approximation method}
   \caption{\texttt{RationalApproximation}}
\begin{algorithmic}[1]
   \STATE {\bfseries Input:} Scalar positive spectrum $M(z)>0$ on  $\TT_N \defeq \{\e^{j 2\pi n /N} \mid n \!=\! 0,\dots,N\!-\!1\}$, and a small positive scalar $\epsilon$
   \vspace{0.5mm}
   \STATE {\bfseries Output:} Causal rational filter $K(z)$ on  $\TT_N$ 
   \vspace{0.5mm}
   \STATE  Get $P(z),Q(z)$  by solving the convex optimization in \eqref{eq:hinf_rational_opt}, for fixed $\epsilon$, given $M(z)$ \\
   \vspace{0.5mm}
   \STATE Get the rational spectral factors of $P(z),Q(z)$, which are $S_P(z),S_Q(z)$ using the canonical Factorization method in \cite{sayed_survey_2001}\\
   \vspace{0.5mm}
   \STATE Get $U^r(z)$,the rational spectral factor of$M(z)$, as $S_P(z)/S_Q(z)$\\
   \STATE Get $K(z)$ from the formulation in \eqref{eq: state space filter}, \eqref{eq:finalstatespaceK}
\end{algorithmic}
\end{algorithm}

\subsection{Proof of \texorpdfstring{\cref{thm:state space filter}}{Theorem 4.6} }

We write the DR estimator, $K(\e^{j\omega})$, as a sum of causal functions:
\begin{align}
      K(\e^{j\omega})&= U^{-1}\{U K_0 \Delta\}_{+}\Delta^{-1} \\
      &=U^{-1} (U\{ K_0 \Delta\}_{+}+ \{U\{ K_0 \Delta\}_{-}\}_{+})\Delta^{-1}\\
      &= \{ K_0 \Delta\}_{+}\Delta^{-1} + U^\inv \{U\{ K_0 \Delta\}_{-}\}_{+}\Delta^{-1}\label{eq:last}
\end{align}

where we drop the dependence of $\Delta, K_0$ and $U$ on $\e^{j\omega}$. 

Given the spectral factor $U(\e^{j\omega})$ in rational form as $U(\ejw)=\Tilde{D}^{1/2}(I+\Tilde{C} (\ejw I -\Tilde{A})^{-1}\Tilde{B})$, its inverse is given by:
\begin{align}\label{eq:linv}
    U^{-1}(\e^{j\omega})=(I-\Tilde{C}(\e^{j\omega}I- (\Tilde{A}-\Tilde{B}\Tilde{C}))^{-1}\Tilde{B})\Tilde{D}^{-1/2}
\end{align}
From the above, we have:
\begin{align}\label{eq:dl}
    \{K_0 \Delta \}_{-}=T(z)= C_sP A^\ast_P (z^\inv I - A^\ast_P )^\inv C_y^\ast
(I + C_yP C_y^\ast)^{-\ast/2}
\end{align}

Multiplying the above equation with $U$, and taking its causal part, we get:
\begin{align}
    \{U\{\Delta K_0\}_{-}\}_{+}= &\{\Tilde{D}^{1/2} C_sP A^\ast_P (z^\inv I - A^\ast_P)^\inv C_y^\ast(I + C_yP C_y^\ast)^{-\ast/2}+ \nonumber \\
    &\Tilde{D}^{1/2}\Tilde{C} (\ejw I -\Tilde{A})^{-1}\Tilde{B} C_sP A^\ast_P (z^\inv I - A^\ast_P )^\inv C_y^\ast(I + C_yP C_y^\ast)^{-\ast/2}\}_{+}
\end{align}

Given that the term $\Tilde{D}^{1/2} C_sP A^\ast_P (z^\inv I - A^\ast_P)^\inv C_y^\ast(I + C_yP C_y^\ast)^{-\ast/2}$ is strictly anticausal, and considering the matrix $U_{ly}$ which solves the lyapunov equation:  
$\Tilde{A}U_{ly} A_P^\ast+\Tilde{B} C_sP A^\ast_P=U_{ly}$, we get $\{U\{ K_0 \Delta\}_{-}\}_{+}$ as:

\begin{align}
&\{U\{\Delta K_0\}_{-}\}_{+}\nonumber \\
&=\{ \Tilde{D}^{1/2}\Tilde{C} \left((z I -\Tilde{A})^{-1}\Tilde{A}U_{ly} + U_{ly} A^\ast_P(z^\inv I - A^\ast_P )^\inv+U_{ly}\right)C_y^\ast(I + C_yP C_y^\ast)^{-\ast/2}\}_{+}\\
&= \Tilde{D}^{1/2}\Tilde{C} \left((z I -\Tilde{A})^{-1}\Tilde{A} +I\right)U_{ly} C_y^\ast(I + C_yP C_y^\ast)^{-\ast/2}\\
&= z\Tilde{D}^{1/2}\Tilde{C} (z I -\Tilde{A})^{-1}U_{ly} C_y^\ast(I + C_yP C_y^\ast)^{-\ast/2}\label{eq:dl_linv}
\end{align}

Now, multiplying equation \eqref{eq:dl_linv} by the inverse of $U$ \eqref{eq:linv}, we get:
\begin{align}
    U^{-1}\{U\{ K_0\Delta\}_{-}\}_{+}&=z (I+\Tilde{C} (\ejw I -\Tilde{A})^{-1}\Tilde{B})^\inv \Tilde{C} (z I -\Tilde{A})^{-1}U_{ly} C_y^\ast(I + C_yP C_y^\ast)^{-\ast/2}\\
    &=z  \Tilde{C} (I+ (z I-\Tilde{A})^\inv \Tilde{B}\Tilde{C})^\inv (z I -\Tilde{A})^{-1}U_{ly} C_y^\ast(I + C_yP C_y^\ast)^{-\ast/2}\\
     &=z  \Tilde{C} (z I-\Tilde{A}_P )^\inv U_{ly} C_y^\ast(I + C_yP C_y^\ast)^{-\ast/2}\\
    &=  \Tilde{C} (I+ \Tilde{A}_P(z I-\Tilde{A}_P )^\inv ) U_{ly} C_y^\ast(I + C_yP C_y^\ast)^{-\ast/2}\\
\end{align}
where $\Tilde{A}_P=\Tilde{A}-\Tilde{B}\Tilde{C}$.

The inverse of $\Delta$ is given by $\Delta^{-1}(z)=(I + C_yP C_y^\ast)^{-1/2}(I - C_y(z I - A_P )^{-1}K_P)$, and we already showed that $  \{ K_0 \Delta\}_{+}= 
C_s(zI -A)^\inv A PC_y^\ast(I+C_yPC_y^\ast)^{-\ast/2} +C_s PC_y^\ast(I+C_yPC_y^\ast)^{-\ast/2}$. 

Then we can get the 2 terms of equation \eqref{eq:last}: 

\begin{equation}\label{eq:1}
    \{ K_0 \Delta\}_{+}\Delta^{-1}=C_sPC_y^\ast (I+C_yPC_y^\ast)^{-1}+C_s \left(I-PC_y^\ast (I+C_yPC_y^\ast)^{-1} C_y\right) (zI-A_P)^\inv K_P
\end{equation}

and
\begin{align}\label{eq:2}
    &U^\inv \{U\{ K_0\Delta\}_{-}\}_{+}\Delta^{-1}
     \nonumber\\
     &=\left(\Tilde{C}U_{ly} C_y^\ast(I + C_yP C_y^\ast)^{-\ast/2}+\Tilde{C}\Tilde{A}_P(z I-\Tilde{A}_P )^\inv  U_{ly} C_y^\ast(I + C_yP C_y^\ast)^{-\ast/2}\right)\nonumber\\
     &\times \left((I + C_yP C_y^\ast)^{-1/2} - (I + C_yP C_y^\ast)^{-1/2}C_y(z I - A_P )^{-1}K_P\right) \\
     &=\Tilde{C}\Tilde{A}_P (z I-\Tilde{A}_P )^\inv U_{ly} C_y^\ast(I + C_yP C_y^\ast)^{-1} \left(I-C_y(z I - A_P )^{-1}K_P)\right)\nonumber\\
     &-\Tilde{C}U_{ly} C_y^\ast(I + C_yP C_y^\ast)^{-1}C_y(z I - A_P )^{-1}K_P \nonumber \\
     &+\Tilde{C}U_{ly} C_y^\ast(I + C_yP C_y^\ast)^{-1}
\end{align}

Finally, summing equations \eqref{eq:1} and \eqref{eq:2}, we get the controller $K(\e^{j\omega})$ in its rational form:
\begin{align}\label{eq:finalstatespaceK}
K(\e^{j\omega})&= \begin{bmatrix}
    \Tilde{C}\Tilde{A}_P & -C_s +C_sPC_y^\ast (I+C_yPC_y^\ast)^{-1} C_y+\Tilde{C}U_{ly} C_y^\ast(I + C_yP C_y^\ast)^{-1}C_y
\end{bmatrix}\\
&\times \left(z I - \begin{bmatrix}
    \Tilde{A}_P & U_{ly} C_y^\ast(I + C_yP C_y^\ast)^{-1}C_y \\0&A_P
\end{bmatrix} \right)^\inv \begin{bmatrix}
    U_{ly} C_y^\ast(I + C_yP C_y^\ast)^{-1}\\-K_P
\end{bmatrix}\\
        &+ \Tilde{C}U_{ly} C_y^\ast(I + C_yP C_y^\ast)^{-1}+C_sPC_y^\ast (I+C_yPC_y^\ast)^{-1}
\end{align}
which can be explicitly rewritten as in equation \eqref{eq: state space filter}, where $\widetilde{F}$,$\widetilde{G}$,$\widetilde{H}$ and $\widetilde{L}$ are defined accordingly.

\end{document}